\documentclass[11pt,reqno]{amsart}
\usepackage{listings}
\usepackage{mathtools}
\usepackage{amssymb}
\usepackage{amsthm}
\usepackage{url}
\usepackage{eucal}
\usepackage{tikz}
\usepackage{cite}
\usetikzlibrary{decorations.markings}
\numberwithin{equation}{section}
\theoremstyle{plain}
\newtheorem{theorem}{Theorem}[section]
\newtheorem{lemma}[theorem]{Lemma}
\newtheorem{proposition}[theorem]{Proposition}
\newtheorem{corollary}[theorem]{Corollary}
\theoremstyle{definition}
\newtheorem{definition}[theorem]{Definition}
\theoremstyle{remark}
\newtheorem{remark}[theorem]{Remark}
\newcommand*{\id}{{\mathrm{id}}}

\newcommand*{\R}{{\mathbb R}}                                                 

\newcommand*{\CC}{{\mathbb{C}}}                                               
                                                
\newcommand*{\Hb}{{\mathbb H}}                                                
                                                
\newcommand*{\Vb}{{\mathbb V}}
\newcommand*{\Wb}{{\mathbb W}}                                                
\newcommand*{\Sb}{{\mathbb S}}                                                
\newcommand*{\Nb}{{\mathbb N}}

\newcommand*{\pf}{{\mathfrak p}}
\newcommand*{\rd}{{\mathrm d}}                                                
                                               
\newcommand*{\Gr}{{\mathrm{Gr}}} 
\newcommand{\pa}{\partial}
\newcommand{\p}{\tilde{P}}
\newcommand{\m}{\mu_1}
\newcommand{\tm}{\tilde{\mu}_1}
\newcommand{\n}{\mu_2}
\newcommand{\tn}{\tilde{\mu}_2}

\allowdisplaybreaks

\begin{document}
\title[Grassmannian flows and integrable systems]{Grassmannian flows
  and applications to non-commutative non-local and local integrable systems}
\author{Anastasia Doikou}
\author{Simon J.A. Malham}
\author{Ioannis Stylianidis}
\address{Maxwell Institute for Mathematical Sciences,        
and School of Mathematical and Computer Sciences,   
Heriot-Watt University, Edinburgh EH14 4AS, UK (7/9/20)}
\email{A.Doikou@hw.ac.uk, S.J.A.Malham@hw.ac.uk, is11@hw.ac.uk}


\begin{abstract}
We present a method for linearising classes of matrix-valued nonlinear partial differential equations
with local and nonlocal nonlinearities. Indeed we generalise a linearisation procedure originally
developed by P\"oppe based on solving the corresponding underlying linear partial differential
equation to generate an evolutionary Hankel operator for the `scattering data', and then solving
a linear Fredholm equation akin to the Marchenko equation to generate the evolutionary solution
to the nonlinear partial differential system. Our generalisation involves inflating the
underlying linear partial differential system for the scattering data to incorporate
corresponding adjoint, reverse time or reverse space-time data, and it also allows
for Hankel operators with matrix-valued kernels. With this approach we show how to linearise
the matrix nonlinear Schr\"odinger and modified Korteweg de Vries equations as well
as nonlocal reverse time and/or reverse space-time versions of these systems. 
Further, we formulate a unified linearisation procedure that incorporates all these
systems as special cases. Further still, we demonstrate all such systems are
example Fredholm Grassmannian flows.
\end{abstract}

\maketitle

\section{Introduction}
The aim of this paper is to formulate a unified programme
for the linearisation of certain types of nonlinear systems.
We use the Grassmann--P\"oppe method presented in Beck \textit{et al.\/} \cite{paper,paper2}
and Doikou \textit{et al.\/} \cite{paper3}. This method combines the idea that linear
flows on Fredholm Stiefel manifolds project onto nonlinear Riccati flows in a given coordinate
patch of corresponding Fredholm Grassmann manifolds, together with the operator approach
developed by P\"oppe~\cite{poppe1,poppe2} for nonlinear integrable systems.
Indeed we streamline P\"oppe's approach, only requiring P\"oppe's kernel product
rule, and also not requiring commutativity of the integral kernels involved.
Further we generalise P\"oppe's approach by generalising the kernel product
rule central to P\"oppe's method by not insisting on evaluating the trace 
in the product rule. The consequences of implementing these generalisations
are that we provide a unified approach to the linearisation of the matrix 
nonlinear Schr\"odinger and modified Korteweg de Vries equations, as well as
the matrix Korteweg de Vries equation itself and matrix nonlocal versions
of these equations as presented in Ablowitz and Musslimani~\cite{AM}. 
By nonlocal we mean the matrix nonlinearities involve factors with time
or space reversal or both. Further, generalising the kernel product rule
means we generate quite general matrix nonlinear partial differential equations
for the underlying integral kernels in the method. These represent a wide class
of integrable systems in themselves.

Let us now try to disentangle these statements as succinctly as possible.
First let us address the Grassmannian and Riccati flow aspect mentioned. 
Central to the method presented in Beck \textit{et al.\/} \cite{paper,paper2}
are a pair of time-dependent Hilbert--Schmidt operators $Q=Q(t)$ and $P=P(t)$.
Suppose this pair of operators satisfy the linear evolutionary system,
\begin{align*}
\partial_t Q &= A(\id+Q) + BP\\
\partial_t P &= C(\id+Q) + DP,
\end{align*}
where $A$ and $C$ are known bounded operators, while $B$ and $D$ are known,
possibly unbounded, operators. Assume a solution $Q=Q(t)$ and $P=P(t)$ exists
to this system, at least up to some time $T>0$. We now introduce a third
Hilbert--Schmidt operator $G=G(t)$ via the relation,
\begin{equation*}
  P=G(\id+Q).
\end{equation*}
We call this the Riccati relation.
Formally by direct straightforward computation,
$G$ evolves according to the Riccati flow
\begin{equation*}
  \pa_t G=C+DG-G(A+BG).
\end{equation*}
In Beck \textit{et al.} \cite{paper,paper2}, we show there exists a unique,
well-behaved solution $G$ to the Riccati relation at least for some time $T>0$,
which indeed evolves according to such a flow. As we explain in detail
in Section~\ref{subsec:Grassmannianflow}, the Riccati flow for $G$
represents a projected flow via the Riccati relation on the canonical
coordinate patch of a Fredholm Grassmann manifold. Further,
since the operators concerned are Hilbert--Schmidt integral operators,
all of the above can be translated to corresponding equations for the kernels of $P,Q$ and $G$.
At the kernel level the Riccati relation has the form,
\begin{equation*}
  p(x,y) = g(x,y)+\int g(x,z)q(z,y)\,\mathrm{d}z.
\end{equation*}
The interval of integration depends on the application.
This is the Marchenko equation which plays a central role in the classical theory of integrable systems,
for example in the Zakharov--Shabat scheme \cite{zs,zs2},
as well as the work of Ablowitz \textit{et al.} \cite{ablowitz}).
Naturally we can interpret the Riccati equation above in terms of
its integral kernel $g=g(x,y;t)$. In Beck \textit{et al.\/} \cite{paper,paper2}
and Doikou \textit{et al.\/} \cite{paper3} we show many classes
of partial differential equations with nonlocal nonlinearites 
can have the representation given by the Riccati equation.
Hence the solutions $g=g(x,y;t)$ to those equations can be
represented by the three linear equations for $Q$, $P$ and $G$;
the linear equation for $G$ is the Riccati relation. In other words
such equations can be linearised. We meet three notions of
nonlocal nonlinearities herein. Example classes of
nonlinear equations we can solve using the approach just mentioned
are as follows. They demonstrate the first notion of 
nonlocal nonlinearities we refer to. One class of
equations for $g=g(x,y;t)$ we solve in Beck \textit{et al.\/} \cite{paper,paper2}
has the nonlocal Korteweg de Vries form, 
\begin{equation*}
\pa_tg(x,y;t)+\pa_x^3g(x,y;t)+\int_{\R}g(x,z;t)\pa_z g(z,y;t)\,\rd z=0.
\end{equation*}
On the other hand in Doikou \textit{et al.\/} \cite{paper3}
we use the method above to solve Smoluchowksi-type coagulation equations
for $g=g(x;t)$ of the form,
\begin{equation*}
\pa_tg(x;t)=d(\pa_x)g(x;t)+\int_0^xg(x-y;t)g(y;t)\,\rd y-g(x;t)\int_0^\infty g(y;t)\,\rd y.
\end{equation*}
Here $d=d(\pa_x)$ is a constant coefficient polynomial of $\pa=\pa_x$.
We associate this class of coagulation nonlocal nonlinear terms with
the notion of nonlocal nonlinearity present in the nonlocal KdV equation just above.

Also in Doikou \textit{et al.\/} \cite{paper3} we generalised the
approach above to consider linear equations for the operators $P$ and $Q$,
and a Riccati relation for $G$, for example as follows:
\begin{align*}
\mathrm{i}\pa_tP&=\pa_x^2P,\\
Q&=P^\dag P,\\
P&=G(\id+Q).
\end{align*}
In addition we assume $P$ is a Hankel operator with an integral kernel
of the form $p=p(y+z+x;t)$. We say this system is linear. This is because,
to solve for $G$, first, we need to solve the linear partial differential equation for $P$.
Then $Q$ is explicity given as a quadratic term in $P$ so we do not
need to solve an equation for $Q$. Then second, we need to solve the
linear Riccati relation for $G$. As shown in Doikou \textit{et al.\/} \cite{paper3},
and we also see in Section~\ref{sec:applications}, in practice
as a result of this procedure we can linearise partial differential
equations for $g=g(y,z;x,t)$ of the form,
\begin{equation*}
\mathrm{i}\pa_tg(y,z;x,t)=\pa_x^2g(y,z;x,t)+2g(y,0;x,t)g^\ast(0,0;x,t)g(0,z;x,t).
\end{equation*}
We call this type of equation a kernel equation, though note the
special form of nonlocal nonlinearity present. This is the second
notion of nonlocal nonlinearity to which we refer. Note if we specialise
this last equation so $y=z=0$, then $g=g(0,0;x,t)$ satisfies
the nonlinear Schr\"odinger equation, with the usual local nonlinearity.
This approach for the kernel form of the nonlinear Schr\"odinger equation
shown just above, is based on the approach P\"oppe~\cite{poppe1,poppe2}
developed. Though the connection with Grassmannian flows can be glimpsed
via the formulation of the linear operator equations above, to maintain
conciseness for now we refer the reader to Section~\ref{subsec:Grassmannianflow}
for the full details. The procedure above lends itself naturally to
the non-commutative setting and, as we see in Section~\ref{sec:applications},
the integral kernels for $P$, $Q$ and $G$ can be matrix-valued.  
Further we have thusfar glossed over an important component of
P\"oppe's method, which is the major insight underlying the method.
This is the aforementioned kernel product rule, in which the Hankel property
of $P$ plays a crucial role. Suppose $F$ is a linear operator with
kernel $f(y,z;x)$. Note we assume $f$ depends on a parameter $x$.
Recall from above if $F$ is a Hankel operator, then $f=f(y+z+x)$.
Let us denote by $[F]$ the kernel of $F$, i.e.\/ $[F]=f$. 
Now suppose $F$, $F'$, $H$ and $H'$ are all Hilbert--Schmidt operators
with continuous kernels, and in addition assume $H$ and $H'$ are
Hankel operators. Then the fundmental theorem of calculus implies
\begin{equation*}
\bigl[F\pa_x(HH')F'\bigr](y,z;x)=[FH](y,0;z)[H'F'](0,z;x).
\end{equation*}
This is the crucial kernel product rule to which we refer.
It generalises the product rule used by P\"oppe who used
the `trace' form of this rule in the sense that the rule was applied
with $y=z=0$. The kernel product rule above is the only property
we assume in Doikou \textit{et al.\/} \cite{paper3} and herein.
Further the reader can now begin to fathom how the nonlinear term
in the kernel equation for $g=g(y,z;x,t)$ was formed---by applying
this product rule twice to the appropriate product of operators.
Indeed one application of the kernel product rule generates 
$[\,\cdot\,](y,0;x)[\,\cdot\cdot\,](0,z;x)$, while after
two applications we get
$[\,\cdot\,](y,0;x)[\,\cdot\cdot\,](0,0;x,t)[\,\cdots\,](0,z;x)$.

Herein we use the generalisations just mentioned to show how
matrix-valued kernel equations analogous to that for $g=g(y,z;x,t)$
above can be linearised. Hence the corresponding standard matrix-valued
local nonlinearity partial differential systems can be linearised.
The matrix-valued systems we establish this for include the nonlinear
Schr\"odinger equation, the modified Korteweg de Vries equation and the
Korteweg de Vries equation. Further we can also linearise in
this manner corresponding matrix-valued nonlocally nonlinear
versions of these equations, including the reverse space-time and 
reverse time nonlocal nonlinear Schr\"odinger equation as well as 
the reverse space-time nonlocal modified Korteweg de Vries equation.
For example the latter equation has the following form for $g=g(x,t)$:
\begin{equation*}
\pa_tg+\pa_x^3g=3g\tilde g (\pa_x g)+3(\pa_xg)\tilde g g,
\end{equation*}
where $\tilde g(x,t)=g^{\mathrm{T}}(-x,-t)$.
This is the third notion of nonlocal nonlinearity to which we refer.
Such equations can be found in Ablowitz and Musslimani~\cite{AM}.
Herein we only focus on the latter two notions of nonlocal nonlinearities.
To distinguish them, we refer to the third notion just above as the
nonlocally nonlinear equations, while the second notion
above involving the kernels $g=g(y,z;x,t)$ we refer to as
nonlocal/kernel or simply `kernel equations'.

Let us now discuss the `unified programme' mentioned at the very beginning.
Herein we extend the Doikou \textit{et al.\/} \cite{paper3}, which is based
on the second system of equations for $P$, $Q$ and $G$ above. Indeed we generalise
and inflate the system as follows, this is the application linear system in
Definition~\ref{appliedlinear}:
\begin{align*}
\pa_t{P}&=\mu_1\pa_x^2 P+\mu_2\pa_x^3 P\\
\pa_t{\p}&=\tilde{\mu}_1\pa_x^2 \p+\tilde{\mu}_2\pa_x^3\p\\
Q&=\p P\\
P&=G(\mathrm{id}+Q),
\end{align*}
where $\mu_1,\mu_2,\tilde{\mu}_1,\tilde{\mu}_2\in\mathbb C$ are constant parameters.
Note we have separate linear equations for the linear operators $P$ and $\p$,
and we also include the linear operators $\tilde Q=P\p$ and $\tilde{G}$,
where $\tilde G$ satisfies $\p=\tilde{G}(\mathrm{id}+\tilde Q)$.
All the matrix nonlocal/kernel and nonlocal nonlinear systems we consider
herein linearise to the system above for appropriate choices of the parameters
and $\p$. For example for the complex matrix modified Korteweg de Vries equation 
we set $\mu_1=\tilde{\mu}_1=0$, $\mu_2=\tilde{\mu}_2=-1$ and $\p=-P^\dag$;
see Remark~\ref{rmk:complexmKdV}. The inflated system above still naturally 
generates a Grassmannian flow; see Section~\ref{subsec:Grassmannianflow}.
Thus all the nonlinear systems we consider herein are examples of such flows,
adding to the large class of nonlocal nonlinear systems
(in the sense of the first notion on nonlocal nonlinearity we mentioned above)
we have identified as such, for example the Smoluchowksi coagulation flows
considered in Doikou \textit{et al.} \cite{paper3}
and all the nonlocal nonlinear systems considered in Beck \textit{et al.} \cite{paper,paper2}.

Our work herein, in Beck \textit{et al.\/} \cite{paper,paper2}
and Doikou \textit{et al.\/} \cite{paper3}, was motivated by the work of
Ablowitz \textit{et al.\/} \cite{ablowitz}, Dyson \cite{dyson},
McKean \cite{mckean} and by a series of papers by P\"oppe \cite{poppe1, poppe2, poppe3},
P\"oppe and Sattinger \cite{poppe4} and Bauhardt and P\"oppe~\cite{BP-ZS}.
Of particular importance for us was the realisation by P\"oppe that the solution to a soliton equation is given
by some function of the Fredholm determinant of the solution to the linearised soliton equation.
That the scattering operators are Hankel operators is another key ingredient in P\"oppe's method.
Hankel operators have received a lot of recent attention, see Grudsky and Rybkin \cite{GR1, GR2},
Grellier and Gerard~\cite{Gerard} and Blower and Newsham \cite{BN}. Non-commutative integrable systems,
see Fordy and Kulisch \cite{FK}, Nijhoff \textit{et al.} \cite{NQLC}, Nijhoff \textit{et al.\/} \cite{NQC},
Fokas and Ablowitz \cite{FA}, Ablowitz, Prinari and Trubatch \cite{APT},
and latterly nonlocal integrable systems have also recently received a lot of attention,
see Ablowitz and Musslimani \cite{AM}, Fokas \cite{F2016} and Grahovski, Mohammed and Susanto \cite{GMS}.
With regard to the non-commutative NLS in particular, see Mumford \cite{M1974} for the first instance
of the multi-component NLS, and for its discretisation, see the work by
Degasperis and Lombardo \cite{DL1, DL2}, Ablowitz \textit{et al.\/} \cite{APT}
and, more recently, Doikou \textit{et al.\/} \cite{DFS} and Doikou and Sklaveniti \cite{DS}. 

Other work in such directions includes Fokas' unified transform method,
see Fokas and Pelloni \cite{FP} and the references therein,
and the scheme by Zakharov and Shabat \cite{zs,zs2}.
Mumford \cite[p.~3.239]{M1984} also took a similar viewpoint
providing solutions to the Sine-Gordon, KdV and KP equations using $\theta$-functions.
Classical integrability involves the existence 
of a Lax pair $(\tilde{L},\tilde{D})$ which satisfies the auxiliary linear problem $\tilde{L}\Psi=\lambda\Psi$
and $\partial_t\Psi = \tilde{D}\Psi$ for an auxiliary function $\Psi$ and spectral parameter $\lambda$.
Requiring these two equations be compatible, one arrives at the so-called zero curvature condition
$\partial_t \tilde{L}=[\tilde{D},\tilde{L}]$, which in turn yields the nonlinear integrable equation.
In this context, the existence of the Lax pair provides extra symmetries and, hence, integrability.
In our formulation, which is based on P\"oppe's method, we impose the linear evolutionary condition
$\partial_t\Psi = \tilde{D}\Psi$ and the Hankel property for $\Psi$ only; 
see Beck \textit{et al.\/} \cite[pp. 5-6]{paper2}. The connection between classical
integrable systems and Grassmannians was first explored by Sato~\cite{SatoI, SatoII}
and developed further by Segal and Wilson~\cite{SW}.

To summarise, what is new in this paper is that, for a collection of classical integrable nonlinear systems, we:
\begin{enumerate}
\item[(i)] Optimise and simplify the method of P\"oppe. We show how a linear equation for
a Hankel operator, based on the linearised version of the system,
together with a linear integral relation, generates the solution
to a corresponding nonlocal/kernel version of the integrable nonlinear systems.
The integrable nonlinear system itself is generated
by a further projection. The approach is optimal/minimal as it only uses the
product rule of the P\"oppe approach, and none further;
\item[(ii)] Generalise this optimal approach to the non-commutative case.
We do not require the operator kernels involved to commute and
thus the entire linearisation procedure for the integrable systems
involved applies to matrix-valued systems;
\item[(iii)] Demonstrate via the non-commutative P\"oppe procedure, 
the nonlocal/kernel nonlinear integrable systems generated are
examples of evolutionary Grassmannian flows. We show
how the evolutionary linear system for the Hankel operator generates an
infinite dimensional Stiefel manifold flow. The projection of that flow
onto the underlying Fredholm Grassmannian in a given coordinate chart
is the nonlocal/kernel nonlinear integrable system under consideration;
\item[(iv)] Show how some matrix-valued nonlocal nonlinear
partial differential systems, including nonlocal reverse time and
reverse space-time versions of classical integrable systems,
can also be linearised using the non-commutative P\"oppe procedure we advocate;
\item[(v)] Reveal how the Miura transformation is the result of
a trivial quadratic operator decomposition into linear operator factors
at the level of the two separate non-commutative linear systems for the modified
Korteweg de Vries and Korteweg de Vries equations;
\item[(vi)] Present a unified programme that incorporates all the matrix
kernel, nonlocal nonlinear and local nonlinear integrable systems
we consider. We show the different systems result from different choices
of the parameters and underlying Hankel operators $P$ and $\p$;
\item[(vii)] Discuss how, for given initial data, this linearisation approach
can be used to find time-evolutionary solutions to such
non-commutative integrable nonlinear systems analytically, or
efficiently numerically. In other words we can generate the
evolutionary solution at any time by evaluating, analytically,
the solution to the underlying linear partial differential equation
for $P$ and if required $\p$ at any given time $t$,
determining $Q$ and if required $\tilde{Q}$ at that time,
and then solving, usually numerically, the linear Fredholm equations defining
$G$ and if required $\tilde{G}$ at that time $t$. We do not have
to numerically evolve the solution to the nonlinear system in time.
This is one of the practical `gains' we have made.
\end{enumerate}  
Our paper is outlined as follows.
In Section~\ref{sec:unification} we introduce some preliminary notions and results we use
throughout this paper, in particular the kernel bracket,
observation functional and product rule.
We then describe the unification scheme in its most general form.
In Section~\ref{sec:applications} we present the main results of this paper.
We start by establishing existence and uniqueness properties and proceed to prove
Theorems \ref{2nd order} and \ref{3rd order},
where the linearisation of both nonlocal/kernel and local versions of different integrable systems
is achieved without assuming commutativity.
Finally, in Section~\ref{sec:discussion} we discuss possible extensions to the work herein.

\section{Unification}\label{sec:unification}

\subsection{Preliminaries}
Let us first describe the general framework by introducing the types of operators
we use and providing some necessary definitions.
As in Doikou \textit{et al.} \cite{paper3}, we consider 
Hilbert--Schmidt integral operators which depend on both a 
spatial parameter $x\in\mathbb{R}$ and a time parameter $t\in[0,\infty)$.
The class of Hilbert--Schmidt operators, $\mathfrak{J}_2$, are representable
in terms of square-integrable kernels, and so given an operator $F=F(x,t)$ with $F\in\mathfrak{J}_2$,
there exists a square-integrable kernel $f=f(y,z;x,t)$ such that
\begin{equation*}
  (F\phi)(y;x,t) = \int_{-\infty}^0 f(y,z;x,t)\phi(z)\,\mathrm{d}z,
\end{equation*}
for any square-integrable function $\phi$.
\begin{definition}[Kernel bracket]
With reference to the operator $F$ just above, we use the \emph{kernel bracket} notation
$[F]$ to denote the kernel of $F$:
\begin{equation*}
  [F](y,z;x,t) \coloneqq f(y,z;x,t).
\end{equation*}
For brevity, we often drop the dependencies and write $[\,\cdot\,]$.
\end{definition}
Of critical importance throughout this paper is a class of integral operators known as
Hankel operators. Indeed we consider time-dependent Hankel operators
which also depend on a parameter $x$ of the following form.
\begin{definition}[Hankel operator with parameter]
  We say a given time-dependent Hilbert--Schmidt operator $H\in\mathfrak{J}_2$
  with corresponding square-integrable kernel $h$ is \emph{Hankel} or \emph{additive}
  with parameter $x\in\mathbb{R}$ if its action, for any square-integrable function $\phi$, is given by
  \begin{equation*}
    (H\phi)(y;x,t) \coloneqq \int_{-\infty}^0 h(y+z+x;t)\phi(z)\,\mathrm{d}z.
  \end{equation*}
\end{definition}
As in P\"oppe \cite{poppe2}, we introduce the following observation functional
for any Hilbert--Schmidt operator with \emph{continuous} kernel.
\begin{definition}[Observation functional]
  Given a Hilbert--Schmidt operator $F$ with a square-integrable and continuous kernel $f=f(y,z)$,
  the \emph{observation functional} $\langle\cdot\rangle$ is defined to be $\langle F\rangle\coloneqq f(0,0)$.
\end{definition}
\begin{remark}
  Further below and in Section~\ref{sec:applications} we are concerned
  with Hilbert--Schmidt operators which depend on both a spatial parameter $x\in\mathbb{R}$
  and a time parameter $t\in [0,\infty)$. In that case, if $F=F(x,t)$ represents such a parameter
  dependent operator with corresponding kernel $f=f(y,z;x,t)$,
  then we have $\langle F(x,t)\rangle=f(0,0;x,t)$. 
\end{remark}
As mentioned previously, there is one key `product rule' property. This is the following;
we include the proof from Doikou \textit{et al.} \cite{paper3} for completeness.
\begin{lemma}[Product rule]
  Assume $H,H'$ are Hankel Hilbert--Schmidt operators with parameter $x$ and $F,F'$ are Hilbert--Schmidt operators.
  Assume further that the corresponding kernels of $H, H', F$ and $F'$ are continuous.
  Then, the following \emph{product rule} holds 
  \begin{equation*}
    [F\pa_x(HH')F'](y,z;x) = [FH](y,0;x)[H'F'](0,z;x).
  \end{equation*}
As a special case, we have
\begin{equation*}
  \langle F\partial_x(HH')F'\rangle = \langle FH\rangle \langle H'F'\rangle.
\end{equation*}
\end{lemma}
\begin{proof}
We use the fundamental theorem
of calculus and Hankel properties of $H$ and $H'$.
Let $f$, $h$, $h'$ and $f'$ denote the integral kernels 
of $F$, $H$, $H'$ and $F'$ respectively.
By direct computation $[F\pa_x(HH')F'](y,z;x)$ equals
\begin{align*}
&\int_{\R_-^3}
f(y,\xi_1;x)\pa_x\bigl(h(\xi_1+\xi_2+x)h^\prime(\xi_2+\xi_3+x)\bigr)
f^\prime(\xi_3,z;x)\,\rd \xi_3\,\rd \xi_2\,\rd \xi_1\\
&=\int_{\R_-^3}
f(y,\xi_1;x)\pa_{\xi_2}\bigl(h(\xi_1+\xi_2+x)h^\prime(\xi_2+\xi_3+x)\bigr)
f^\prime(\xi_3,z;x)\,\rd \xi_3\,\rd \xi_2\,\rd \xi_1\\
&=\int_{\R_-^2}
f(y,\xi_1;x)h(\xi_1+x)h^\prime(\xi_3+x)f^\prime(\xi_3,z;x)\,\rd \xi_3\,\rd \xi_1\\
&=\int_{\R_-}f(y,\xi_1;x)h(\xi_1+x)\,
\rd \xi_1\cdot\int_{\R_-}h^\prime(\xi_3+x)f^\prime(\xi_3,z;x)\,\rd \xi_3\\
&=\bigl([FH](y,0;x)\bigr)\bigl([H'F'](0,z;x)\bigr),
\end{align*}
giving the first result. Setting $y=z=0$ generates the second result.
\end{proof}
\begin{remark}\label{motivation}
  Our motivation for introducing the observation functional $\langle\cdot\rangle$
  and the product rule above stems for the work of P\"oppe in \cite{poppe1, poppe2, poppe3}
  and P\"oppe and Sattinger \cite{poppe4}.
  In our case, however, we develop the product rule for the kernel bracket $[\,\cdot\,]$
  as a precursor to the observation functional $\langle\,\cdot\,\rangle$.
  Each application of the product rule produces an extra degree of nonlinearity,
  hence we do so in a targeted way so as to generate the nonlinearity required.
  The fact this is accomplished at the kernel level means the resulting nonlinear
  partial differential equation is a nonlocal/kernel one in the sense described
  in the Introduction. Having generated the nonlocal/kernel PDE, specialising
  to the observation functional $\langle\cdot\rangle$ has the effect
  of passing to the corresponding local PDE, as every term is now evaluated at $(x,t)$.
\end{remark}
We record in the following lemma some identities for $U\coloneqq (\mathrm{id}+F)^{-1}$
which are useful later on. Here we assume $F$ depends on a parameter.
Similar results are derived by P\"oppe \cite{poppe1, poppe2}.
\begin{lemma}[\textbf{U-identities}]\label{Uid}
Suppose the operator $F$ depends on a parameter with respect
to which we wish to compute derivatives. Further suppose $U\coloneqq (\mathrm{id}+F)^{-1}$ exists.
Then the following identities hold:
\begin{enumerate}
\item[(i)] $\partial U = -U(\partial F)U$;
\item[(ii)] $\mathrm{id} - U = UF = FU$;
\item[(iii)] $U_x = -U_xF - UF_x = -F_xU - FU_x$;
\item[(iv)] $U_{xx} = -2U_xF_xU - UF_{xx}U = -2UF_xU_x - UF_{xx}U$;
\item[(v)] $U_{xxx} = -6U_xF_xU_x - 3UF_{xx}U_x - 3U_xF_{xx}U - UF_{xxx}U$;
\item[(vi)] $UF_xU_x = U_xF_xU$.
\end{enumerate}
\end{lemma}
\begin{proof}
(i) Since $UU^{-1}=\mathrm{id}$, we have $(\partial U)U^{-1}=-U\partial F$
and so $\partial U = -U(\partial F) U$.
(ii) We have $\mathrm{id} - U=U(U^{-1}-\mathrm{id})=(U^{-1}-\mathrm{id})U$ which gives the
result since $U^{-1}-\mathrm{id} = F$.
(iii) These follow by differentiating (ii).
(iv) From the first part of (iii) we have $U_{xx}=-U_{xx}F-2U_xF_x-UF_{xx}$,
so $U_{xx}(\mathrm{id}+F)=-2U_xF_x-UF_{xx}$ and hence $U_{xx} = -2U_xF_xU - UF_{xx}U$.
Similarly, $U_{xx} = -2UF_xU_x - UF_{xx}U$.
(v) Again from the first part of (iii) we have $U_{xx}=-U_{xx}F-2U_xF_x-UF_{xx}$,
so $U_{xxx}=-U_{xxx}F-3U_{xx}F_x-3U_xF_{xx}-UF_{xxx}$ and thus
$U_{xxx}(\mathrm{id}+F)=-3U_{xx}F_x-3U_xF_{xx}-UF_{xxx}$.
This implies $U_{xxx}=-3U_{xx}F_xU-3U_xF_{xx}U-UF_{xxx}U$.
Now substituting $U_{xx} = -2U_xF_xU - UF_{xx}U$ from (iv) and using
(i) gives the result.
(vi) Using (i), we have $UF_xU_x = -UF_xUF_xU = U_xF_xU$.
\end{proof}

\subsection{Unification scheme}
Assume $P=P(x,t)$ and $\p=\p(x,t)$ are Hankel Hilbert--Schmidt operators with
respective integral kernels $p=p(y+z+x;t)$ and $\tilde{p}=\tilde{p}(y+z+x;t)$
and $Q=Q(x,t)$ and $G=G(x,t)$ are Hilbert--Schmidt operators
with respective kernels $q=q(y,z;x,t)$ and $g=g(y,z;x,t)$.
\begin{definition}[Abstract linear system]\label{def:abstractlinearsystem}
Assume the operators $P$, $\p$, $Q$ and $G$ satisfy the linear system of equations:
\begin{equation*}
\begin{aligned}
\partial_t{P} &= d(\partial_x)P\\
\partial_t{\p} &= \tilde{d}(\partial_x)\p\\
Q &= \p P\\
P &= G(\mathrm{id} + Q),
\end{aligned}
\end{equation*}
where $d$ and $\tilde{d}$ are polynomials of $\partial_x$ with constant coefficients. 
In addition, for convenience later, we also set $\tilde Q=P\p$ and
define $\tilde{G}$ as the solution to the linear equation $\p=\tilde{G}(\mathrm{id}+\tilde Q)$.
\end{definition}
As in Beck \textit{et al.} \cite{paper, paper2},
our goal later is to extract a nonlinear flow for $G$ from this linear system.
Then, after considering the corresponding kernels, we apply the product rule whenever necessary,
so as to show the function $[G]=[G](y,z;x,t)$ satisfies in each case the targeted nonlocal/kernel integrable system.
The different cases are distinguished by the order of $d$ and $\tilde{d}$ and, crucially, by the choice of $\p$.
A further projection using the observation functional yields
the analogous local integrable system for $\langle G\rangle=\langle G\rangle(x,t)$.
As this whole process does not require commutativity between the operator kernels involved,
the systems generated can have matrix form. 

These ideas, besides being related to Zakharov--Shabat scheme \cite{zs,zs2},
Fokas and Pelloni \cite{FP} and Ablowitz \textit{et al.} \cite{APT}
in the general context of inverse scattering, have been motivated
P\"oppe \cite{poppe1, poppe2, poppe3}, P\"oppe and Sattinger \cite{poppe4}
and Ablowitz \textit{et al.} \cite{ablowitz}. In particular, the generation of the nonlinear PDEs
from the linear system above via algebraic manipulations and the product rule is similar
in spirit to P\"oppe \cite{poppe1, poppe2}.
Moreover, the corresponding assumptions in Ablowitz \textit{et al.} \cite{ablowitz}
have been instructive in the way we choose the operator $\p$ to generate each case.

Lastly, in light of the system of linear equations we introduced above we have the following
corollary to Lemma~\ref{Uid}.
\begin{corollary}\label{cor:Vid}
  Suppose we set $F\coloneqq Q$ with $Q=\p P$ and $U\coloneqq (\mathrm{id}+Q)^{-1}$
  so $U=(\mathrm{id} + \p P)^{-1}$.  Assume $U$ exists.
  Then $U$ satisfies properties (i)--(vi) in Lemma~\ref{Uid}.
  Further suppose we set $F\coloneqq\tilde Q$ with $\tilde Q=P\p$ and $V\coloneqq(\mathrm{id}+\tilde Q)^{-1}$.
  Assume $V$ exists. Then similarly $V$ satisfies properties (i)--(vi) in Lemma~\ref{Uid}.
  In addition we note $PU^{-1}=V^{-1}P$, so we have $VP=PU=G$.
  Similarly, we have $U\p=\p V=\tilde{G}$.  
\end{corollary}

\subsection{Grassmannian flow}\label{subsec:Grassmannianflow}
We now outline how the flow prescribed in Definition~\ref{def:abstractlinearsystem} represents
a flow on a Fredholm Grassmann manifold. Given the fields $P$, $\p$, $Q$, $\tilde Q$, $G$
and $\tilde G$ satisfying the linear system of equations in Definition~\ref{def:abstractlinearsystem},
we set:
\begin{equation*}
  \mathcal P\coloneqq\begin{pmatrix} P &\p\end{pmatrix}, \quad
  \mathcal Q\coloneqq\begin{pmatrix} Q & O\\ O&\tilde Q\end{pmatrix}
  \quad\text{and}\quad
  \mathcal G\coloneqq\begin{pmatrix} G &\tilde G\end{pmatrix}.
\end{equation*}
We observe the system of linear equations satisfied by $G$ and $\tilde G$ is equivalent to
the linear equation:
\begin{equation*}
  \mathcal P=\mathcal G\,(\mathrm{id}+\mathcal Q).
\end{equation*}
If $\mathcal P$ and $\mathcal Q$ are a Hilbert--Schmidt operators and $(\mathrm{id}+\mathcal Q)^{-1}$
exists, then the solution $\mathcal G$ represents an element in a given coordinate patch
of a Fredholm Grassmann manifold, or ``Grassmannian'' for short.
We briefly explain this context here. More details and background information can be found in 
Beck \textit{et al.\/ }~\cite{paper,paper2}, Doikou \textit{et al.} \cite{paper3},
Segal and Wilson~\cite{SW} and Pressley and Segal~\cite{PS}.
Suppose we have a separable Hilbert space $\Hb$.
The Fredholm Grassmannian of all subspaces of $\Hb$ that are
comparable in size to a given closed subspace $\Vb\subset\Hb$
is defined as follows; see Pressley and Segal~\cite{PS}. 
\begin{definition}[Fredholm Grassmannian]\label{def:FredholmGrassmannian}
Let $\Hb$ be a separable Hilbert space with a given decomposition
$\Hb=\Vb\oplus\Vb^\perp$, where $\Vb$ and $\Vb^\perp$ are infinite
dimensional closed subspaces. The Grassmannian $\Gr(\Hb,\Vb)$
is the set of all subspaces $\Wb$ of $\Hb$ such that:
\begin{enumerate}
\item[(i)] The orthogonal projection $\mathrm{pr}\colon\Wb\to\Vb$ is 
a Fredholm operator, indeed it is a Hilbert--Schmidt perturbation
of the identity; and
\item[(ii)] The orthogonal projection $\mathrm{pr}\colon\Wb\to\Vb^\perp$ 
is a Hilbert--Schmidt operator.
\end{enumerate}
\end{definition}
To help elucidate the structure of the Grassmannian $\Gr(\Hb,\Vb)$ and
respresentative coordinate charts, let us construct an example Grassmannian,
suitably general to apply to our applications herein. More details
can be found in Doikou \textit{et al.} \cite{paper3} for example.
Any separable Hilbert space is isomorphic to the sequence space
of square summable complex sequences $\ell^2(\CC)$.
We can also assume the sequences in $\ell^2(\CC)$ are parametrised by $\Nb$.
Thus we can represent any sequence $\mathfrak a\in\ell^2(\CC)$ by an
infinite vector $\mathfrak a=(\mathfrak a(1),\mathfrak a(2),\mathfrak a(3),\ldots)^{\mathrm{T}}$ 
where $\mathfrak a(n)\in\CC$ for every $n\in\Nb$.
Square summability of any sequence $\mathfrak a$ means $\sum_{n\in\Nb}\mathfrak a^\ast(n)\mathfrak a(n)<\infty$,
where $\mathfrak a^\ast(n)$ denotes the complex conjugate of $\mathfrak a(n)$.
For any two elements $\mathfrak a, \mathfrak b\in\ell^2(\CC)$,
the natural inner product on $\ell^2(\CC)$ is given by $\mathfrak a^\dag\mathfrak b$,
where $\mathfrak a^\dag$ denotes the complex conjugate transpose of the vector $\mathfrak a$.
Hence a natural canonical orthonormal basis for $\ell^2(\CC)$ are the vectors
$\{\mathfrak e_n\}_{n\in\Nb}$ where $\mathfrak e_n$ is the infinite vector 
with one in the $n^{\mathrm{th}}$ position and zeros elsewhere.
We now assume the underlying separable Hilbert space $\Hb$ and closed subspace $\Vb$
are $\Hb\coloneqq\ell^2_{\mathrm{l}}(\CC)\times\ell^2_\mathrm{r}(\CC)$,
where $\ell^2_{\mathrm{l}}(\CC)$ and $\ell^2_\mathrm{r}(\CC)$ are just independent
copies of $\ell^2(\CC)$, and $\Vb\coloneqq\ell^2(\CC)$.
Suppose we are given a set of independent sequences in $\ell^2_{\mathrm{l}}(\CC)\times\ell^2_\mathrm{r}(\CC)$
which span $\ell^2(\CC)$ and we record them as columns in the infinite matrix
\begin{equation*}
  W=\begin{pmatrix}\id+\mathcal Q\\\mathcal P
  \end{pmatrix}.
\end{equation*}
In other words, each column of $\id+\mathcal Q\in\ell^2_{\mathrm{l}}(\CC)$ 
and each column of $\mathcal P\in\ell^2_{\mathrm{r}}(\CC)$. Assume also for the moment, when we
constructed $\id+\mathcal Q$ we ensured it was a Fredholm operator on $\ell^2_{\mathrm{l}}(\CC)$ 
with $\mathcal Q\in\mathfrak J_2(\ell^2_{\mathrm{l}}(\CC);\ell^2_{\mathrm{l}}(\CC))$. Here 
$\mathfrak J_2(\ell^2(\CC);\ell^2(\CC))$ is the class of Hilbert--Schmidt operators
from $\ell^2(\CC)$ to $\ell^2(\CC)$, equipped with the norm
$\|\mathcal Q\|^2_{\mathfrak J_2}\coloneqq\mathrm{tr}\,\mathcal Q^\dag\mathcal Q$
where `$\mathrm{tr}$' is the trace operator. We also denote by
$\mathfrak J_1(\ell^2(\CC);\ell^2(\CC))$ the class of trace-class operators.
For any operator $\mathcal Q\in\mathfrak{J}_n$, $n=1,2$, we can define
\begin{equation*}
  \mathrm{det}_n(\mathrm{id}+\mathcal Q)
  \coloneqq\exp\Bigg(\sum_{k\geqslant n} \frac{(-1)^{k-1}}{k}\mathrm{tr}(\mathcal Q^k)\Bigg).
\end{equation*} 
This is the Fredholm determinant when $n=1$, or the regularised Fredholm determinant when $n=2$.
For more details, see Simon \cite{S2005}. The operator $\id+\mathcal Q$ is invertible
if and only if $\mathrm{det}_n(\mathrm{id}+\mathcal Q)\neq0$.
We also assume when we constructed $\mathcal P$ we ensured
$\mathcal P\in\mathfrak J_2(\ell^2_{\mathrm{l}}(\CC);\ell_{\mathrm{r}}^2(\CC))$,
the space of Hilbert--Schmidt operators from
$\ell^2_{\mathrm{l}}(\CC)$ to $\ell^2_{\mathrm{r}}(\CC)$.
With this in hand, we denote by $\Wb$ the subspace of
$\ell^2_{\mathrm{l}}(\CC)\times\ell^2_{\mathrm{r}}(\CC)$ spanned by the
columns of $W$. Further let $\Vb_0$ denote the canonical 
subspace with the representation
\begin{equation*}
  V_0=\begin{pmatrix}\id\\ O
  \end{pmatrix},
\end{equation*}
where $\id=\id_{\Vb_0}$ and $O$ is the infinite matrix of zeros.
Now consider the projection of $\Wb$ onto $\Vb_0$. The projections
$\mathrm{pr}\colon\Wb\to\Vb_0$ and $\mathrm{pr}\colon\Wb\to\Vb_0^\perp$ respectively give
\begin{equation*}
  W^\parallel=\begin{pmatrix}\id+\mathcal Q\\O
  \end{pmatrix}
  \quad\text{and}\quad
  W^\perp=\begin{pmatrix}O\\\mathcal P
  \end{pmatrix}.
\end{equation*}
This projection is achievable if and only if $\mathrm{det}_2(\mathrm{id}+\mathcal Q)\neq0$. 
We assume this is the case for the moment, and discuss what happens when this is
not the case momentarily. Hence we see the subspace of $\Hb$ spanned by the columns
of $W^\parallel$ coincides with the subspace spanned by the columns of $V_0$ which is $\Vb_0$.
Indeed the transformation $(\id+\mathcal Q)^{-1}\in\mathrm{GL}(\Vb)$
transforms $W^\parallel$ to $V_0$. Under this transformation,
the representation $W$ for $\Wb$ becomes
\begin{equation*}
  \begin{pmatrix}\id\\\mathcal G
  \end{pmatrix},
\end{equation*}
where $\mathcal G=\mathcal P(\id+\mathcal Q)^{-1}$. Any subspace that can
be projected onto $\Vb_0$ can be represented in this way and vice-versa.
Indeed operators $\mathcal G\in\mathfrak J_2(\ell^2_{\mathrm{l}}(\CC);\ell_{\mathrm{r}}^2(\CC))$
parameterise all the subspaces $\Wb$ that can be projected onto $\Vb_0$.
See Segal and Wilson~\cite{SW} and Pressley and Segal~\cite{PS} for more details.
When $\mathrm{det}_2(\mathrm{id}+\mathcal Q)=0$, and $\Wb$ cannot be projected
onto $\Vb_0$, then $\mathrm{dim}(\Wb\cap\Vb_0^\perp)>0$, and we need to choose
a different representative coordinate patch. Given a subset $\Sb=\{i_1,i_2,\ldots\}\subset\Nb$,
let $\Vb_0(\Sb)$ denote the subspace given by
$\mathrm{span}\{\mathfrak e_{i_1},\mathfrak e_{i_2},\ldots\}$.
From Pressley and Segal~\cite[Prop.~7.1.6]{PS} we know there exists
a set $\Sb\subset\Nb$ such that the projection $\Wb\to\Vb_0(\Sb)$
is an isomorphism. Hence we can carry through the procedure above
with $\Vb_0$ replaced by $\Vb_0(\Sb)$,
representing a different representative coordinate chart.
For more details see Doikou \textit{et al.} \cite{paper3}. For more details
on the implications concerning the connection between singularities
develeping in components of $\mathcal G$, poor representative patches
and the need to choose a different representative coordinate patch,
see Beck \textit{et al.\/ }~\cite{paper,paper2} and the Discussion Section~\ref{sec:discussion}.

With reference to the abstract linear system in Definition~\ref{def:abstractlinearsystem}
we require the operators $P$ and $\p$ and thus their corresponding
kernels to satisfy the partial differential systems involving the
operators $d=d(\pa_x)$ and $\tilde d=\tilde d(\pa_x)$. These are
both assumed to be constant coefficient polynomials of $\pa_x$. 
Looking forward to our actual application linear system in Definition~\ref{appliedlinear}
and results in Theorems~\ref{2nd order} and \ref{3rd order}, 
in each example case we in fact have $\mathrm{deg}(d)=\mathrm{deg}(\tilde d)$.
Assume this is the case for the present discussion.
Let us now consider the Grassmannian construction above in
the context of the applications we consider.
Let us temporarily pull back from the inflated system for $\mathcal P$, $\mathcal Q$
and $\mathcal G$ to the system of equations for $P$, $\p$, $Q$, $\tilde Q$, $G$
and $\tilde G$. We assume the operators $P$ are Hilbert--Schmidt operators
from $L^2(\R;\CC^m)$ to $L^2(\R;\CC^n)$ whose kernels are sufficiently smooth
for the differential equations to make sense. Indeed we suppose $P\in\mathrm{Dom}(d)$,
where $\mathrm{Dom}(d)$ is the subspace of $\mathfrak J_2(L^2(\R;\CC^m);L^2(\R;\CC^n))$
for which the kernels $p$ of the operators $P$ therein are such that
$p\in H^{\mathrm{deg}(d)}(\R^2;\CC^{n\times m})$.
Analogously we suppose $\p\in\mathrm{Dom}(\tilde d)$,
where $\mathrm{Dom}(\tilde d)$ is the subspace of $\mathfrak J_2(L^2(\R;\CC^n);L^2(\R;\CC^m))$
for which the kernels $\tilde p$ of the operators $\p$ therein are such that
$\tilde p\in H^{\mathrm{deg}(d)}(\R^2;\CC^{m\times n})$.
Since $Q=\p P$ and $\tilde Q=P\p$, i.e.\/ they are both the product of two Hilbert--Schmidt operators,
we thus know they are both trace class. We thus assume $Q$ and $\tilde Q$ lie in
the subspaces of $\mathfrak J_1(L^2(\R;\CC^m);L^2(\R;\CC^m))$ and $\mathfrak J_1(L^2(\R;\CC^n);L^2(\R;\CC^n))$,
respectively, for which their corresponding kernels are such that
$q\in H^{\mathrm{deg}(d)}(\R^2;\CC^{m\times m})$ and
$\tilde q\in H^{\mathrm{deg}(d)}(\R^2;\CC^{n\times n})$.
Then, for example, for the inflated system we equivalently assume
$\mathcal P\in\mathrm{Dom}(d)\times\mathrm{Dom}(\tilde d)$ and so forth. 
For appropriate initial data, the following lemma establishes the
existence and uniqueness for some interval of time, of solutions
$G$ and $\tilde G$ to the respective Fredholm equations $P=G(\id+Q)$
and $\p=\tilde G(\id+\tilde Q)$, and thus to $\mathcal P=\mathcal G(\id+\mathcal Q)$.
Recall $\mathfrak J_2$ is a subspace of $\mathfrak J_1$.
A proof can be found in Doikou \textit{et al.\/} \cite{paper3}. 
\begin{lemma}[\textbf{Existence and Uniqueness}]\label{existuniq}
  Assume, for some $T>0$, we know $Q\in C^\infty([0,T];\mathfrak{J}_2(\mathbb{V};\mathbb{V}))$
  and $P\in C^\infty([0,T];\mathrm{Dom}(d))$.  Also assume $\det_2(\id+Q_0)\neq0$.
  Then, there exists a $T'>0$, with $T'\leqslant T$, such that $\det_2(\id+Q(t))\neq0$ for $t\in[0,T']$.
  In particular, there exists a unique solution $G$ to the linear Fredholm equation $P=G(\mathrm{id}+Q)$
  with $G\in C^\infty([0,T'];\mathrm{Dom}(d))$.  
\end{lemma}

\section{Application to integrable PDEs}\label{sec:applications}

\subsection{Existence and Uniqueness results}
Before we can proceed to the derivation of the nonlinear PDEs from the linear system,
we need to establish some existence and uniqueness results.
To begin we introduce some notation we require. 
For $w:\R\to\mathbb{R}_+$, we denote by $L^2_w(R;\CC^{n\times m})$
the space of functions $f:\R\to\CC^{n\times m}$ whose $L^2$ norm weighted by $w$ is finite, that is,
\begin{equation*}
  \|f\|_{L^2_w} \coloneqq\mathrm{tr}\,\int_{\R} f^\dag(x)f(x)w(x)\,\mathrm{d}x < \infty,
\end{equation*} 
where $f^\dag$ denotes the complex-conjugate transpose of $f$ and `$\mathrm{tr}$' is the trace
operator giving the sum of the diagonal elements of a matrix. We set 
\begin{equation*}
  H(\R;\CC^{n\times m})\coloneqq \underset{l\in\{0\}\cup\mathbb{N}}\bigcap H^l(\R;\CC^{n\times m}),
\end{equation*}
to denote the Sobolev space of all functions which themselves, as well as their derivatives,
are square integrable. Any such functions are thus smooth. 
Further, for any given integrable function $f=f(x)$,
we denote its Fourier transform by $\mathfrak f=\mathfrak f(k)$
and define its inverse by the formulae 
\begin{equation*}
\mathfrak f(k)\coloneqq\int_{\mathbb R} f(x)\mathrm{e}^{2\pi\mathrm{i}kx}\,\mathrm{d} x
\qquad\text{and}\qquad
f(x)\coloneqq\int_{\mathbb R} \mathfrak f(k)\mathrm{e}^{-2\pi\mathrm{i}kx}\,\mathrm{d} k.
\end{equation*}
We now establish regularity results for the linear PDEs satisfied
by the integral kernels $p$ and $\tilde{p}$, of $P$ and $\p$ respectively, i.e.\/ for 
\begin{equation*}
\begin{aligned}
\pa_t{p} &= d(\pa_x)p\\
\pa_t\tilde{p} &= \tilde{d}(\pa_x)\tilde{p}.
\end{aligned}
\end{equation*}
We assume $d=d(\pa_x)$ and $\tilde{d}=\tilde{d}(\pa_x)$ are scalar polynomials of $\pa_x$
satisfying $(d(\mathrm{i}\kappa))^\ast=-d(\mathrm{i}\kappa)$ for all $\kappa\in\R$.
The regularity results ensure these kernels generate Hilbert--Schmidt operators $P$ and $\p$, respectively.
The original proof of the following lemma for the scalar case can be found in Doikou \textit{et al.} \cite{paper3}.
We include the matrix adapted version below for completeness.
\begin{lemma}[\textbf{Dispersive linear PDE properties}]\label{disp}
Assume $p=p(x;t)$ is a solution to the general dispersive linear partial differential equation above.
Let $w=w(\partial)$ denote an arbitrary polynomial function of $\partial=\partial_x$
with constant non-negative coefficients, whose Fourier transform
we denote by $\mathfrak{w}=\mathfrak{w}(k)$, while $W:\mathbb{R}\to\mathbb{R}_+$ denotes
the specific function $W:x\mapsto 1+x^2$.
Then, $p$ and its Fourier transform $\mathfrak{p}=\mathfrak{p}(k;t)$
satisfy the following properties for all $k\in\mathbb{R}$ and $t\geqslant0$
($d'$ below is the derivative of $d$):
\begin{enumerate}
\item[(i)] $\|\partial\mathfrak{p}\|_{L^2}^2 \leqslant (2\pi)^2\|W^{1/2}p\|_{L^2}^2;$
\item[(ii)] $p(0)\in H(\mathbb{R};\CC^{n\times m}) \Rightarrow \mathfrak{p}(0)\in L_\mathfrak{w}^2(\mathbb{R};\CC^{n\times m});$
\item[(iii)] $p(0)\in L_W^2(\mathbb{R};\CC^{n\times m}) \Rightarrow \mathfrak{p}(0)\in H^1(\mathbb{R};\CC^{n\times m});$
\item[(iv)] $\mathfrak{p}(k;t) = e^{td(2\pi ik)}\mathfrak{p}(k;0);$
\item[(v)] $\mathfrak{p}^\dag(k;t)\mathfrak{p}(k;t) = \mathfrak{p}^\dag(k;0)\mathfrak{p}(k;0);$
\item[(vi)] $\|w(\partial)p(t)\|_{L^2}^2 = \|\mathfrak{w}\mathfrak{p}(t)\|_{L^2}^2
  = \|\mathfrak{w}\mathfrak{p}(0)\|_{L^2}^2 = \|w(\partial)p(0)\|_{L^2}^2;$
\item[(vii)] $p(0)\in H(\mathbb{R};\CC^{n\times m}) \Rightarrow p(t)\in H(\mathbb{R};\CC^{n\times m});$
\item[(viii)] $\|\partial\mathfrak{p}(t)\|_{L^2}^2 \leqslant
  2((2\pi)^2t^2\|d'\mathfrak{p}(0)\|_{L^2}^2 + \|\partial\mathfrak{p}(0)\|_{L^2}^2)$;
\item[(ix)] $\mathfrak{p}(0)\in H^1(\mathbb{R};\CC^{n\times m})\cap L_{(d')^2}^2(\mathbb{R};\CC^{n\times m})
  \Rightarrow \mathfrak{p}(t)\in H^1(\mathbb{R};\CC^{n\times m});$
\item[(x)] $\|W^{1/2}p(t)\|_{L^2}^2 = \|p(0)\|_{L^2}^2 + (2\pi)^{-2}\|\partial\mathfrak{p}(t)\|_{L^2}^2;$
\item[(xi)] $p(0)\in L_W^2(\mathbb{R};\CC^{n\times m})\cap H^{\mathrm{deg}(d')}(\mathbb{R};\CC^{n\times m})
  \Rightarrow p(t)\in L_W^2(\mathbb{R};\CC^{n\times m});$
\item[(xii)] $p(t)\in L_W^2(\mathbb{R};\CC^{n\times m})$ $\Rightarrow$ Hankel $P(t)\in\mathfrak{J}_2$.
\end{enumerate}
\end{lemma}
\begin{proof}
Let $V\colon\R\to\R_+$ denote the function $V\colon x\mapsto x^2$.
We establish the results in order as follows: 
(i) By the Plancherel Theorem and the definition of the Fourier transform
we observe $\|\pa\pf\|_{L^2}^2=(2\pi)^2\|V^{1/2}p\|_{L^2}^2$ which we 
then combine with the fact $\|V^{1/2}p\|_{L^2}^2\leqslant\|W^{1/2}p\|_{L^2}^2$;
(ii) This follows again by the Plancherel Theorem and standard properties
of the Fourier transform;
(iii) This follows from (i) applied at time $t=0$;
(iv) and (v) These follow by directly solving the linear differential
equation for the general dispersive equation in Fourier space;
(vi) The first and third equalities follow by the Plancherel theorem, 
while the second uses (v); (vii) Follows from (vi) and that $w$ is arbitrary;
(viii) The derivative with respect to $k$ of the explicit solution from (iv)
generates 
$\pa\pf(t)=\mathrm{e}^{td(2\pi\mathrm{i}k)}(t(2\pi\mathrm{i})d^\prime\pf(k;0)+\pa\pf(k;0))$,
where $d^\prime$ denotes the derivative of $d$. Taking the complex conjugate
of this and using both expressions to expand $\pa\pf^\dag(t)\pa\pf(t)$
generates the inequality shown when we integrate with respect to $k$ and 
use the Cauchy--Schwarz and Young inequalities;
(ix) Follows from (viii); (x) We observe 
$\|W^{1/2}p(t)\|_{L^2}^2
=\|p(t)\|_{L^2}^2+\|V^{1/2}p(t)\|_{L^2}^2=\|p(0)\|_{L^2}^2+(2\pi)^{-2}\|\pa\pf(t)\|_{L^2}^2$,
where we used the equality stated in the proof of (i) just above;
(xi) This follows from (iii), (ix) and (x); and finally (xii)
By a standard change of variables $\xi=y+z$ and $\eta=y-z$ 
we have (see for example Power~\cite{Power})
\begin{align*}
\int_{-\infty}^0\int_{-\infty}^0p^\dag(y+z;t)p(y+z;t)\,\rd y\,\rd z
&=\int_{-\infty}^0\int_{\xi}^{-\xi}p^\dag(\xi;t)p(\xi;t)\,\rd \eta\,\rd\xi\\
&=\int_{-\infty}^0(2|\xi|)p^\dag(\xi;t)p(\xi;t)\,\rd\xi.
\end{align*}
Taking the trace, the right-hand side is bounded by $\|W^{1/2}p(t)\|_{L^2}^2$.
Hence we see if $p(t)\in L^2_W(\R;\CC^{n\times m})$ then we observe Hankel $P(t)\in\mathfrak J_2$.
\end{proof}
Finally, we establish the sense in which a solution generated by the linear system exists.
\begin{lemma}[\textbf{Existence and Uniqueness: unified PDE presciption}]\label{existuniqpde}
Assume $p_0\in H(\R;\CC^{n\times m})\cap L_W^2(\R;\CC^{n\times m})$,
$\tilde{p}_0\in H(R;\CC^{m\times n})\cap L_W^2(\R;\CC^{m\times n})$
and $\det(\mathrm{id}+Q(x;0))\neq0$. Then, there exists a $T>0$ such that, for each $t\in[0,T]$
and $x\in\mathbb{R}$, we have:

(i) The solutions $p=p(y+x;t)$ and $\tilde{p}=\tilde{p}(y+x;t)$ to
the respective linear partial differential equations $\partial_t p=d(\partial)p$
and $\partial_t\tilde{p}=\tilde{d}(\partial)\tilde{p}$,
are such that $p(\cdot+x;t)\in H(\R;\CC^{n\times m})\cap L_W^2(\R;\CC^{n\times m})$
and $\tilde{p}(\cdot+x;t)\in H(\R;\CC^{m\times n})\cap L_W^2(\R;\CC^{m\times n})$
with $p(x;0)=p_0(x)$ and $\tilde{p}(x;0)=\tilde{p}_0(x)$. Thus $P(x;t)\in\mathfrak{J}_2$
and $\p(x;t)\in\mathfrak{J}_2$ and are smooth functions of $x$ and $t$;

(ii) The kernel function corresponding to $Q$ given by
\begin{equation*}
q(y,z;x,t) = \int_{-\infty}^0 \tilde{p}(y+\xi+x;t)p(\xi+z+x;t)\,\mathrm{d}\xi,
\end{equation*} 
is such that $Q(x,t)\in\mathfrak{J}_1$ and is a smooth function of $x$ and $t$;

(iii) $\det(\mathrm{id}+Q(x;t))\neq0$;

(iv) There exists a unique $g\in C^\infty([0,T];C^\infty(\R_-^{\times2}\times\R;\CC^{n\times m}))$
which satisfies the linear Fredholm equation 
\begin{equation*}
  p(y+z+x;t) = g(y,z;x,t) + \int_{-\infty}^0 g(y,\xi;x,t)q(\xi,z;x,t)\,\mathrm{d}\xi.
\end{equation*}
\end{lemma}
\begin{proof}
  (i) The time regularity of $p$ follows from the spatial regularity assumed
  on the initial data $p_0$ and Lemma \ref{disp} (iv). The regularity of $p=p(\cdot+x;t)$
  with respect to $x$ follows from Lemma \ref{disp} (vii) and the Hankel assumption.
  Thus, from Lemma \ref{disp} (xi) and (xii), we deduce $P(x;t)\in\mathfrak{J}_2$
  and is a smooth function of $x$ and $t$. The same arguments apply for $\tilde{p}$ and $\p$.
  (ii) Since $P(x;t),\p(x;t)\in\mathfrak{J}_2$, by the Hilbert--Schmidt ideal property
  we have $\|\p P\|_{\mathfrak{J}_1}\leqslant\|\p\|_{\mathfrak{J}_2}\|P\|_{\mathfrak{J}_2}$,
  and hence $Q=\p P\in\mathfrak{J}_1$ for every $x\in\mathbb{R}$ and $t\in[0,T]$.
  (iii) Since $P(x;t),\p(x;t)$ are smooth in $x,t$ so is $Q(x;t)$. Hence, since $\det(\mathrm{id}+Q(x;0))\neq0$,
  there exists $T'>0$ such that $\det(\mathrm{id}+Q(x;t))\neq0$ for $t\in[0,T']$; if $T'<T$ we reset $T$ to be $T'$.
  (iv) This is established using the corresponding abstract result
  in the Existence and Uniqueness Lemma~\ref{existuniq}
  and noting $Q=\p P\in\mathfrak{J}_1\subset\mathfrak{J}_2$.
\end{proof}

\subsection{Matrix nonlinear Schr\"odinger system}\label{NLS}
We now restrict our choice of operators $d$ and $\tilde{d}$. Consider
the following ``application'' linear system.
\begin{definition}[Application linear system]\label{appliedlinear}
Suppose the linear operators $P$, $\p$, $Q$ and $G$ satisfy the linear system of equations,
\begin{align*}
\pa_t{P}&=\mu_1\pa_x^2 P+\mu_2\pa_x^3 P\\
\pa_t{\p}&=\tilde{\mu}_1\pa_x^2 \p+\tilde{\mu}_2\pa_x^3\p\\
Q&=\p P\\
P&=G(\mathrm{id}+Q),
\end{align*}
where the constant parameters $\mu_1,\mu_2,\tilde{\mu}_1,\tilde{\mu}_2\in\mathbb C$.
Recall we also set $\tilde Q=P\p$ and define $\tilde{G}$
as the solution to the linear equation $\p=\tilde{G}(\mathrm{id}+\tilde Q)$.
\end{definition}
With regard to the parameters $\mu_j$ and $\tilde{\mu}_j$, a priori these can be regarded as arbitrary complex numbers.
However, as it turns out, the structure of the computations needed to prove the following two theorems
is such that we eventually require $\tilde{\mu}_j=\pm\mu_j$. For the moment
we also distinguish between the second-order and third-order cases. Thus herein and in Section~\ref{mKdV}
we impose $\mu_1\mu_2=0$. As we show below, special values of $\mu_j$ and $\tilde{\mu}_j$
yield specific well-known integrable systems in these cases. The case when both $\mu_1$ and
$\mu_2$ are non-zero is treated in Section~\ref{sec:mixandmatch}.

We now assume $\mu_2=\tilde{\mu_2}=0$. In Section~\ref{mKdV} we consider the case $\m=\tm=0$. 
We state and prove our first result which leads to the matrix NLS. Recall Corollary~\ref{cor:Vid}. 
\begin{theorem}[\textbf{Second-order decomposition}]\label{2nd order}
Assume the Hilbert--Schmidt operators $P$, $\p$, $Q$ and $G$ satisfy the application linear system
in Definition \ref{appliedlinear} and their corresponding kernels
satisfy the assumptions of Lemma~\ref{existuniqpde}.
Set $\p=P^\dagger$, the adjoint of the operator $P$, and $\m=-\mathrm{i}$.
Then, for some $T>0$, the integral kernel $g=g(y,z;x,t)$ corresponding to $G$,
for every $t\in[0,T]$ satisfies the matrix kernel NLS equation:
\begin{equation*}
\mathrm{i}\partial_t g(y,z;x,t) = \partial_x^2g(y,z;x,t) + 2g(y,0;x,t)g^\dag(0,0;x,t)g(0,z;x,t),
\end{equation*}
where here $g^\dag$ now denotes the complex conjugate transpose of the matrix $g$.
In particular, $\langle G\rangle(x,t)\coloneqq g(0,0;x,t)$ satisfies the matrix NLS equation:
\begin{equation*}
\mathrm{i}\pa_t\langle G\rangle
=\pa_x^2\langle G\rangle+2\langle G\rangle\langle G\rangle^\dag\langle G\rangle.
\end{equation*}
\end{theorem}
\begin{proof}
We split the proof into three steps.\smallskip
  
\emph{Step 1: Apply the linear dispersion operator ro $G=PU$.}
With $G=PU$, using the Leibniz rule, that $P_t=\mu_1 P_{xx}$ and $\p_t=\tm\p_{xx}$,
and the identities $\pa U=-U(\pa Q)U$ and $U_{xx}=-2U_xQ_xU-UQ_{xx}U$
from Lemma~\ref{Uid} in (i) and (iv) with $F=Q$, we compute
\begin{align*}
\pa_t G-\mu_1\pa_x^2 G
&=P_tU-PUQ_tU-\mu_1\bigl(P_{xx}U+2P_xU_x+PU_{xx}\bigr)\\
&=-PU(Q_t-\mu_1Q_{xx})U+2\mu_1\bigl(P_xUQ_xU+PU_xQ_xU\bigr).
\end{align*}
We now set $\tm=-\mu_1$ and assume this holds hereafter.
Since $Q\coloneqq\p P$, by direct computation we have
\begin{align*}
Q_t-\mu_1 Q_{xx}&=\tm\p_{xx}P+\m\p P_{xx}-\m\bigl(\p_{xx}P+2\p_x P_x+\p P_{xx}\bigr)\\
&=(\tm-\m)\p_{xx}P-2\m\p_xP_x\\
&=-2\m(\p_xP)_x.
\end{align*}
Substituting this result into the previous one and using $Q=\p P$, we get,
\begin{equation*}
\partial_t G - \mu_1\partial_x^2 G = 2\m\big(PU(\p_xP)_xU + P_xU(\p P)_xU + PU_x(\p P)_xU\big).
\end{equation*}\smallskip

\emph{Step 2: Apply the kernel bracket operator.}
We now consider the corresponding kernels, denoted by $[\cdot]$. In other words
applying the kernel bracket to the final relation in Step~1 above, and using the
kernel bracket product rule, we find,
\begin{align*}
\partial_t[G](y,z) &- \m\partial_x^2[G](y,z)\\
&= 2\m\big([PU\p_x](y,0)[PU](0,z) + [P_xU\p](y,0)[PU](0,z)\\
&\quad + [PU_x\p](y,0)[PU](0,z)\big)\\
&= 2\m[(PU\p)_x](y,0)[PU](0,z).
\end{align*}
Now, the identity $\mathrm{id} - V = VP\p = PU\p$, gives $V_x = -(PU\p)_x$.
But also, by the definition of $V$, we have $V_x = -V(P\p)_xV$.
Hence, applying the kernel bracket product rule once more we get,
\begin{align*}
\partial_t[G](y,z) - \m\partial_x^2[G](y,z) &= 2\m[V(P\p)_xV](y,0)[PU](0,z)\\
&= 2\m[VP](y,0)[\p V](0,0)[PU](0,z)\\
&= 2\m[G](y,0)[\tilde{G}](0,0)[G](0,z).
\end{align*}\smallskip

\emph{Step 3: Choose $\p$ and $\m$ for the matrix kernel NLS equation.}  
We now choose $\p=P^\dagger$ and $\m=-\mathrm{i}$, so $\tm=\mathrm{i}$,
where $P^\dagger$ denotes the adjoint of the complex-valued Hilbert--Schmidt operator $P$.
This choice is consistent with the earlier choice of $\tm=-\m$ and the linear partial
differential equation for $\p$.
In this case, $V^\dagger=V$ and hence, $\tilde{G}=P^\dagger V=(VP)^\dagger=G^\dagger$.
The operator $G^\dagger$ has kernel $g^\dag(z,y;x,t)$, where the former is the
adjoint of the operator $G$ and the latter is the
complex conjugate transpose of its matrix kernel $g$.
Hence we conclude $g$ satisfies the matrix kernel NLS equation,
\begin{equation*}
  \mathrm{i}\partial_t g(y,z;x,t) = \partial_x^2g(y,z;x,t) + 2g(y,0;x,t)g^\dag(0,0;x,t)g(0,z;x,t).
\end{equation*}
Further, the function $\langle G\rangle=\langle G\rangle(x,t)$, where
$\langle G\rangle(x,t)=g(0,0;x,t)$, satisfies the matrix local NLS equation
stated in the theorem.
\end{proof}
\begin{remark}\label{rmk:choicetildep}
Suppose at the beginning of Step~3 in the proof of Theorem~\ref{2nd order} just above,
we kept the parameter $\m\in\CC$ general though made the choice $\p=-P^\dag$ instead.  
For this choice we observe $V=(\id-PP^\dag)^{-1}=V^\dag$ and so $\tilde{G}=-P^\dag V=-(VP)^\dag=-G^\dag$.
If we now substitute this form for $\tilde{G}$ into the final relation in Step~2 of
the proof of Theorem~\ref{2nd order} we obtain the following matrix kernel equation,
\begin{equation*}
  \pa_t g(y,z;x,t)=\m\pa_x^2g(y,z;x,t)-2\m g(y,0;x,t)g^\dag(0,0;x,t)g(0,z;x,t).
\end{equation*}
The choice $\m=-\mathrm{i}$ generates the corresponding matrix
kernel NLS equation to that in Theorem~\ref{2nd order} but with a different
sign for the nonlinear term. 
\end{remark}
Making further different consistent choices for $\p$ generates
matrix versions of the reverse space-time nonlocal NLS equation
and the reverse time nonlocal NLS equation given in Ablowitz and Musslimani~\cite[Eq.'s~(5),~(6)]{AM}.
Indeed we actually generate the matrix kernel versions of these equations. 
\begin{corollary}[Reverse space-time matrix nonlocal NLS equation]
If we choose $\p(x,t)=P^{\mathrm{T}}(-x,-t)$, where $P^{\mathrm{T}}$ is the operator
whose matrix kernel is the transpose of the matrix kernel corresponding to $P$,
and $\m=-\mathrm{i}$, then the integral kernel $g=g(y,z;x,t)$ corresponding to $G$,
for every $t\in[0,T]$ satisfies the reverse space-time matrix nonlocal kernel NLS equation:
\begin{equation*}
  \mathrm{i}\pa_t g(y,z;x,t)=\pa_x^2g(y,z;x,t)+2g(y,0;x,t)g^{\mathrm{T}}(0,0;-x,-t)g(0,z;x,t).
\end{equation*}
Setting $y=z=0$ generates the reverse space-time matrix nonlocal NLS equation.
Similarly if we choose $\p(x,t)=P^{\mathrm{T}}(x,-t)$ and $\m=-\mathrm{i}$, then
the integral kernel $g=g(y,z;x,t)$ corresponding to $G$ satisfies the
corresponding reverse time matrix nonlocal kernel NLS equation. And
setting $y=z=0$ generates the reverse time matrix nonlocal NLS equation.
\end{corollary}
\begin{proof}
Recall with $\m=-\mathrm{i}$ the operator $P=P(x,t)$ satisfies the linear
PDE $\pa_tP=-\mathrm{i}\pa_x^2P$ while $\p=\p(x,t)$ satisfies $\pa_t\p=-\mathrm{i}\pa_x^2\p$.
Note the choice $\p(x,t)=P^{\mathrm{T}}(-x,-t)$ is consistent with these two 
linear PDEs. Recall $G=PU$ while $\tilde G=\p V$ where $U=(\id-\p P)^{-1}$
and $V=(\id-P\p)^{-1}$. We observe, substituting for $\p(x,t)=P^{\mathrm{T}}(-x,-t)$, we have
\begin{equation*}
\tilde G(x,t)=P^{\mathrm{T}}(-x,-t)\bigl(\id+P(x,t)P^{\mathrm{T}}(-x,-t)\bigr)^{-1},
\end{equation*}
while,
\begin{equation*}
G(-x,-t)=P(-x,-t)\bigl(\id+P^{\mathrm{T}}(x,t)P(-x,-t)\bigr)^{-1}.
\end{equation*}
It is evident $\tilde G(x,t)=G^{\mathrm{T}}(-x,-t)$ which gives the result
if we substitute this form for $\tilde G$ into the final relation in Step~1 in
the proof of Theorem~\ref{2nd order}.
The reverse time matrix nonlocal kernel NLS equation follows immediately.
\end{proof}
\begin{remark}[Coupled diffusion/anti-diffusion system]
Suppose in Theorem~\ref{2nd order} and in particular in the final relation
in Step~2 of the proof of Theorem~\ref{2nd order}, we instead set
$\p=P^{\mathrm{T}}(x,-t)$ and $\m=1$, so $\tm=-1$, which is a consistent choice.
Then a straightforward computation reveals
$\tilde{G}(x,t)=\p(x,t) V(x,t)=U^{\mathrm{T}}(x,-t)P^{\mathrm{T}}(x,-t)=G^{\mathrm{T}}(x,-t)$.
So we conclude the kernels $g$ and $\tilde{g}$ satisfy the matrix kernel system,
\begin{align*}
\pa_t g(y,z;x,t)&=\pa_x^2g(y,z;x,t)+2g(y,0;x,t)\tilde{g}(0,0;x,t)g(0,z;x,t),\\
\pa_t\tilde{g}(y,z;x,t)&=-\pa_x^2\tilde{g}(y,z;x,t)-2\tilde{g}(y,0;x,t)g(0,0;x,t)\tilde{g}(0,z;x,t).
\end{align*}
Further, the functions $\langle G\rangle=\langle G\rangle(x,t)$
and $\langle\tilde G\rangle=\langle\tilde G\rangle(x,t)$, where
$\langle G\rangle(x,t)=g(0,0;x,t)$ and $\langle\tilde{G}\rangle(x,t)=\tilde{g}(0,0;x,t)$,
satisfy the coupled matrix-valued diffusion/anti-diffusion system with cubic nonlinearity, 
\begin{align*}
\pa_t\langle G\rangle&=\pa_x^2\langle G\rangle+2\langle G\rangle\langle\tilde{G}\rangle\langle G\rangle,\\
\pa_t\langle\tilde{G}\rangle&=-\pa_x^2\langle\tilde{G}\rangle
-2\langle\tilde{G}\rangle\langle G\rangle\langle\tilde{G}\rangle.
\end{align*}
Retrospectively examining the construction of these solutions, 
we require $p$ and $\p$ to satisfy the linear equations
$\pa_t p=\pa_x^2 p$ and $\pa_t\tilde{p}=-\pa_x^2\tilde{p}$. In particular they do not
satisfy the assumptions required for Lemma~\ref{disp}. The solution $p$ satisfies
the heat equation which is well-posed, while the solution $\tilde{p}$ satisfies
the backward heat equation. The latter has a solution via Fourier transform,
but loses exponentially weighted Fourier regularity with time.
\end{remark}

\subsection{Korteweg de Vries and modified equation}\label{mKdV}
We now assume $\m=\tm=0$.
We formulate and derive the corresponding result for the KdV and mKdV equations.
\begin{theorem}[\textbf{Third-order decomposition}]\label{3rd order}
Assume the Hilbert--Schmidt operators $P$, $\p$, $Q$ and $G$ satisfy the application linear system
in Definition~\ref{appliedlinear}
and their corresponding kernels satisfy the assumptions of Lemma~\ref{existuniqpde}.
Then, for some $T>0$, the integral kernel $g=g(y,z;x,t)$ corresponding to $G$, for every $t\in[0, T]$ satisfies:

\noindent (i) When $\p=-P^{\mathrm{T}}$ and $\n=-1$, the matrix kernel mKdV equation:
\begin{align*}
\partial_t g(y,z;x,t) + \partial_x^3g(y,z;x,t)=&\;
3g(y,0;x,t)g^{\mathrm{T}}(0,0;x,t)\partial_x g(0,z;x,t)\\
&\;+3(\partial_x g(y,0;x,t))g^{\mathrm{T}}(0,0;x,t)g(0,z;x,t).
\end{align*} 
In particular, $\langle G\rangle(x,t)\coloneqq g(0,0;x,t)$ satisfies the matrix mKdV equation:
\begin{equation*}
\partial_t\langle G\rangle + \partial_x^3\langle G\rangle
=3\langle G\rangle\langle G\rangle^{\mathrm{T}} \partial_x\langle G\rangle
+3(\partial_x\langle G\rangle)\langle G\rangle^{\mathrm{T}}\langle G\rangle;
\end{equation*} 

\noindent (ii) When $\p=-\mathrm{id}$ and $\n=-1$,
the primitive form of the square-matrix kernel KdV equation:
\begin{equation*}
\pa_tg(y,z;x,t)+\pa_x^3g(y,z;x,t)=3\pa_x g(y,0;x,t)\pa_x g(0,z;x,t).
\end{equation*}
In particular, $\langle G\rangle(x,t)\coloneqq g(0,0;x,t)$
satisfies the primitive form of the square-matrix KdV equation:
\begin{equation*}
\partial_t\langle G\rangle+\partial_x^3\langle G\rangle=3(\partial_x\langle G\rangle)^2.
\end{equation*}
\end{theorem} 
\begin{proof}
Recall $G=PU=VP$. We set $\n$ and $\tn$ equal to each other momentarily.
We split the proof into the following steps.\smallskip

\emph{Step 1: Apply the linear dispersion operator to $G=PU$.}
With $G=PU$, using the Leibniz rule, that $P_t=\n P_{xxx}$
and the identities for $\pa U$, $U_{xx}$
and $U_{xxx}$ from Lemma~\ref{Uid} in (i), (iv) and (v), we compute
\begin{align*}
\partial_t G - \n&\partial_x^3 G\\
=&\; P_tU -PUQ_tU- \n(P_{xxx}U + 3P_{xx}U_x + 3P_xU_{xx} + PU_{xxx})\\
=&\; -PU(Q_t-\n Q_{xxx})U + \n\bigl(3P_{xx}UQ_xU+ 6P_xU_xQ_xU\\
&\;+3P_xUQ_{xx}U + 6PU_xQ_xU_x + 3PUQ_{xx}U_x + 3PU_xQ_{xx}U\bigr).
\end{align*}
We now set $\tn=\n$ and assume this holds hereafter.
Since $Q\coloneqq\p P$, by direct computation we have
\begin{align*}
Q_t-\n Q_{xxx}=&\;\tn\p_{xxx}P + \n\p P_{xxx}\\
&\;-\n\bigl(\p_{xxx}P+3\p_{xx}P_x+3\p_xP_{xx}+\p P_{xxx}\bigr)\\
=&\;- 3\n(\p_{x}P_x)_x.
\end{align*}
Substituting this result into the previous one and dividing by $3\n$, we find
\begin{align*}
\tfrac13({\n}^{-1}\partial_t G - \partial_x^3 G)
=&\;PU(\p_{x}P_x)_xU+P_{xx}UQ_xU+2P_xU_xQ_xU\\
&\;+P_xUQ_{xx}U + 2PU_xQ_xU_x + PUQ_{xx}U_x \\
&\;+ PU_xQ_{xx}U.
\end{align*}\smallskip

\emph{Step 2: Apply the kernel bracket operator.}
We note $Q_x=(\p P)_x$ and $Q_{xx}=(\p_x P)_x+(\p P_x)_x$,
From the final relation in Step~1 we isolate the term
$PUQ_{xx}U_x=PU(\p_x P)_xU_x+PU(\p P_x)_xU_x$. Focusing
on the second term on the right, using $U_x=-UQ_xU$,
applying the kernel bracket and using the product rule, we see,
\begin{equation*}
[PU(\p P_x)_xU_x]=-[PU(\p P_x)_xU(\p P)_xU]=-[PU(\p P_x)_xU\p][PU].
\end{equation*}
The remaining terms in the final relation in Step~1, including
the first term on the right just above, are straightforward.
Using the formulae for $Q_x$ and $Q_{xx}$ and applying the 
kernel bracket with its product rule, we find, 
\begin{align*}
\tfrac13({\n}^{-1}\partial_t [G] - \partial_x^3 [G])
=&\;[PU\p_{x}][P_xU]+[P_{xx}U\p][PU]+2[P_xU_x\p][PU]\\
&\;+[P_xU\p_x][PU]+[P_xU\p][P_xU] + 2[PU_x\p][PU_x]\\
&\;+[PU\p][PU_x]-[PU(\p P_x)_xU\p][PU] \\
&\;+[PU_x\p_x][PU]+[PU_x\p][P_xU].
\end{align*}\smallskip

\emph{Step 3: Collate terms with postfactors $[PU_x]$ and $[P_xU]$.}
We recall from Lemma~\ref{Uid}(ii) with $F=\tilde Q=P\p$ and $V$ in place of $U$,
that $\id-V=\tilde QV=P\p V=PU\p$ since $U\p=\p V$ from Corollary~\ref{cor:Vid}.
Hence $V_x=-(PU\p)_x$. Note also from Lemma~\ref{Uid}(i) we have $V_x=-V\tilde Q_xV=-V(P\p)_xV$.
The terms with postfactors $[PU_x]$ and $[P_xU]$ in the final
relation in Step~2, using these identities and $U_x=-UQ_xU=-U(\p P)_xU$, are
\begin{align*}
\bigl([PU&\p_{x}]+[P_xU\p]+[PU_x\p]\bigr)[P_xU]
+2[PU_x\p][PU_x]+[PU\p][PU_x]\\
=&\;[(PU\p)_x][P_xU]+[(PU\p)_x][PU_x]+[PU_x\p][PU_x]-[P_xU\p][PU_x]\\
=&\;-[V_x][G_x]+[PUQ_xU\p][PUQ_xU]+[P_xU\p][PUQ_xU]\\
=&\;[VP][\p V][G_x]+[PU\p]^3[G]+[P_xU\p][PU\p][G].
\end{align*}
\smallskip

\emph{Step 4: Collate the prefactors of $[PU]$.}
Consider all the terms in the final relation in Step~2 which premultiply
the postfactors $[PU]$. All these terms, without the kernel bracket for
the moment, are:
\begin{align*}
P_{xx}U\p+&2P_xU_x\p+P_xU\p_x-PU(\p P_x)_xU\p+PU_x\p_x\\
(a) =&\;P_{xx}U\p-P_xU(\p P)_xU\p+P_xU\p_x-PU\p_x P_xU\p\\
&\;-PU\p P_{xx}U\p+P_xU_x\p+PU_x\p_x\\
(b) =&\;P_{xx}U\p-P_xU\p_x PU\p-P_xU\p P_xU\p+P_xU\p_x\\
&\;-PU\p_x P_xU\p -PU\p P_{xx}U\p+P_xU_x\p+PU_x\p_x\\
(c) =&\;VP_{xx}U\p-P_xU\p P_xU\p+P_xU\p_xV-PU\p_x P_xU\p\\
&\;+P_xU_x\p+PU_x\p_x\\
(d) =&\;VP_{xx}\p V-PU_x\p_xV+V_xP\p_xV+VP_x\p_xV\\
&\;-(PU\p)_xP_xU\p+(PU_x\p)P_xU\p+P_xU_x\p+PU_x\p_x\\
(e) =&\;VP_{xx}\p V+V_xP\p_xV+VP_x\p_xV+V_xP_x\p V\\
&\;+(PU_x\p)P_xU\p+P_xU_x\p+PU_x\p_x(\id-V)\\
(f) =&\;V(P_{x}\p)_xV+V_x(P\p)_xV+PU_x\p P_xU\p\\
&\;+P_xU_x\p+PU_x\p_xPU\p\\
(g) =&\;V(P_{x}\p)_xV+V_x(P\p)_xV+PU_x(\p P)_xU\p\\
&\;-P_xU(\p P)_xU\p.
\end{align*}
In the computation above, in:
(a) We split the term $2P_xU_x\p$ and used $U_x=-UQ_xU=-U(\p P)_xU$;
(b) We split the term $P_xU(\p P)_xU\p$ using the product rule;
(c) We recalled from Step~3, $\id-V=PU\p$, and made this replacement
in each of the two possible places, leading to cancellation of some terms;
(d) We used $U\p=\p V$, since $PU=VP$ substituted $P_xU=-PU_x+V_xP+VP_x$ in the term $P_xU\p_xV$
and then replaced the two negative terms with postfactors $P_xU\p$ by $(PU\p)_xP_xU\p-(PU_x\p)P_xU\p$;
(e) We combined the two terms as shown and used $V_x=-(PU\p)_x$;
(f) We use the identity $\id-V=PU\p$ and finally in (g) We combined the
terms shown and used $U_x=-U(\p P)_xU$.
If now re-introduce the kernel bracket to the terms in the last
computation and use the kernel bracket product rule, we see,
\begin{align*}
[P_{xx}U\p+&2P_xU_x\p+P_xU\p_x-PU(\p P_x)_xU\p+PU_x\p_x]\\
=&\;[VP_x][\p V]+[V_xP][\p V]+[PU_x\p][PU\p]-[P_xU\p][PU\p]\\
=&\;[(VP)_x][\p V]-[PU\p]^3-[P_xU\p][PU\p],
\end{align*}
where in the last step we used $U_x=-U(\p P)_xU$ and applied the
kernel bracket product rule once more. \smallskip

\emph{Step 5: Combine Steps 3 and 4.} We add the final terms on the right-hand
side in the relation in Step~3 to the terms in the final relation in Step~4,
with the latter terms postmultipled by $[PU]$ as we were only considering
the prefactors of $[PU]$ in Step~4. Using $G=PU=VP$ and recall we also set
$\tilde G=\p V=U\p$, this gives,
\begin{align*}
&[VP][\p V][G_x]+[PU\p]^3[G]+[P_xU\p][PU\p][G]\\
&+[(VP)_x][\p V][PU]-[PU\p]^3[PU]-[P_xU\p][PU\p][PU]\\
&\qquad\qquad\qquad\qquad\qquad\qquad\qquad\qquad\qquad=[G][\tilde G][G_x]+[G_x][\tilde G][G].
\end{align*}
Hence we have shown 
\begin{equation*}
\tfrac13({\n}^{-1}\pa_t[G]-\pa_x^3[G])
=[G][\tilde G][G_x]+[G_x][\tilde G][G].
\end{equation*}
We now choose $\p=-P^{\mathrm{T}}$. By this we mean $\p$ corresponds to the operator whose
kernel is minus the matrix transpose of the kernel corresponding to $P$. We observe
$\tilde G=-G^{\mathrm{T}}$ since $U^{\mathrm{T}}=U$. Hence if $\n=-1$, so $\tn=-1$,
then the equation above corresponds to the matrix kernel mKdV equation stated
in part (i) of the theorem.
In particular, $\langle G\rangle=\langle G\rangle(x,t)$ satisfies the local matrix mKdV equation
also stated in part (i) of the theorem.\smallskip

\emph{Step 6: The KdV derivation.}
For the KdV equation we set $\p=-\mathrm{id}$ and $\n=-1$.
Since $\p$ is now not Hankel, we return to the final relation in Step~1.
Since $\p=-\mathrm{id}$ so $Q=-P$, the final relation in Step~1 becomes,
\begin{align*}
-\tfrac13({\n}^{-1}\partial_t G - \partial_x^3 G)
=&\;P_{xx}UP_xU+2P_xU_xP_xU+P_xUP_{xx}U\\
&\;+2PU_xP_xU_x+PUP_{xx}U_x+PU_xP_{xx}U.
\end{align*}
Since $\p=-\mathrm{id}$ we have $U=V=(\id-P)^{-1}$ so
$G=PU=UP$. Hence the third and sixth terms
on the right above combine as follows:
\begin{align*}
P_xUP_{xx}U+PU_xP_{xx}U&=(PU)_xP_{xx}U\\
&=(UP)_xP_{xx}U\\
&=U_xPP_{xx}U+UP_xP_{xx}U.
\end{align*}
Since $U=(\id-P)^{-1}$, we observe $U_x=UP_xU$. Further we observe 
$\id-U=-PU=-UP$ and thus $U_x=(PU)_x=(UP)_x$. With these in mind,
we see the first and fifth terms on the right above combine as follows:
\begin{align*}
P_{xx}UP_xU+PUP_{xx}U_x&=P_{xx}U_x-(\id-U)P_{xx}U_x\\
&=UP_{xx}U_x\\
&=UP_{xx}P_xU+UP_{xx}PU_x.
\end{align*}
Finally the second and fourth terms on the right above combine as follows:
\begin{align*}
2P_xU_x&P_xU+2PU_xP_xU_x\\
(a)&=2P_xU_xU_x-2P_xU_xPU_x+2U_xP_xU_x-2P_xUP_xU_x\\
(b)&=2P_xU_xU_x-2P_xU_xU_x+2U_xP_xU_x\\
(c)&=U_xP_xU_x+U_xP_xU_x\\
(d)&=U_xPP_xU_x+UP_xP_xU_x+U_xP_xPU_x+U_xP_xP_xU.
\end{align*}
In the computation above, in: (a) We used $U_x=P_xU+PU_x$ twice;
(b) We combined the second and fourth terms using the same identity again in the middle of the terms;
(c) We cancelled like terms and split the remaining term and finally
(d) We used the same identity again on the first and then final factors, respectively, of
the first and second terms.
We now combine the first through sixth terms together, so the
first relation in this step becomes 
\begin{align*}
-\tfrac13({\n}^{-1}\partial_t G - \partial_x^3 G)
=&\;U_x(PP_x)_xU+U(P_xP_x)_xU\\
&\;+U(P_xP)_xU_x+U_x(PP)_xU_x.
\end{align*}
If we now apply the kernel bracket and its product rule, and note since $\id-U=-PU=-UP$ we have
$U_x=G_x$ and $U_x=(PU)_x=(UP)_x=P_xU+PU_x=U_xP+UP_x$, then we find
\begin{align*}
-\tfrac13({\n}^{-1}\partial_t [G] - \partial_x^3 [G])
=&\;[U_xP][P_xU]+[UP_x][P_xU]\\
&\;+[UP_x][PU_x]+[U_xP][PU_x]\\
=&\;[U_x][U_x]\\
=&\;[G_x][G_x].
\end{align*}
This corresponds to the primitive form of the square-matrix kernel KdV equation stated in part (ii)
of the theorem since $\n=-1$. In particular, $\langle G\rangle=\langle G\rangle(x,t)$ satisfies
the primitive form of the local square-matrix KdV equation stated in part (ii) of the theorem.
\end{proof} 
\begin{remark}[Complex mKdV]\label{rmk:complexmKdV}
In the mKdV derivation above the operators $P$, $\tilde P$ and thus $G$ and $\tilde G$
could be complex-valued. Setting $\tilde P=-P^\dag$ generates complex versions of the kernel mKdV
and mKdV equations in Theorem~\ref{3rd order}. Note when $\tilde P=-P^\dag$ then
$V=V^\dag$ and $\tilde{G}=-G^\dag$. See Remark~\ref{rmk:choicetildep}.
\end{remark}
\begin{remark}[KdV derivation]
The derivation of the KdV equation given in Step~6 of the proof of Theorem~\ref{3rd order}
above relied on the final relation in Step~1. The derivation can be shortened by utilising
earlier on in Step~1, since in this case $\id-U=-PU=-UP$ and therefore we know $U_x=(PU)_x$
and also $U_x=UP_xU$, we can compute $\pa_x^3G=\pa_x^3(PU)=\pa_x^2(U_x)=\pa_x^2(UP_xU)$
and so forth. This is the approach used in Doikou \textit{et al.\/} \cite{paper3}.
\end{remark}
\begin{remark}[Kernel bracket application]
We applied the kernel bracket in Step~2 of the proof. This keeps the computation relatively succinct.
However the computation can be developed much further before necessarily applying the kernel bracket.
This means imposing the specialisation $\p=-\id$ can be delayed until a penultimate step. 
See Styliandis~\cite{Stylianidis} for more details. 
\end{remark}
\begin{remark}[Linearisation]\label{rmk:linearisation}
We emphasise Theorems~\ref{2nd order} and \ref{3rd order} establish that the
application linear system given in Definition~\ref{appliedlinear}, for
the choices of the parameters outlined, represents a linearisation of the
time-evolutionary integrable nonlinear PDEs considered. The solutions
to the linear partial differential equations for $P$ and $\p$, or in particular their kernels,
represents the first linear system we need to solve, which is achievable analytically.
We then compute the operators $Q$ and $\tilde Q$ directly from $P$ and $\p$ via their definitions
involving the respective products of these operators. Then the second system of linear equations
we need to solve are the linear Fredholm equations for $G$ and $\tilde G$.
\end{remark}
\begin{remark}[Solving initial value problems]\label{rmk:IVPs}
Concerning the initial value problem for each of the equations
of Theorems \ref{2nd order} and \ref{3rd order}, in light of the Remark~\ref{rmk:linearisation},
we work as follows.
Given arbitrary smooth initial data $p_0=p_0(x)$ and $\tilde{p}_0=\tilde{p}_0(x)$,
we solve the linear equations for $p$ and $\tilde{p}$ in Definition~\ref{appliedlinear}
analytically via Fourier transform or convolutional integrals. 
We then find $q$ by evaluating the integral in Lemma \ref{existuniqpde}(ii).
Finally, we generate the solution $[G]=g(y,z;x,t)$ to the kernel PDE
by solving the linear integral equation of Lemma~\ref{existuniqpde}(iv).
Setting $y=z=0$ yields the solution $\langle G\rangle=g(0,0;x,t)$ to the corresponding local PDE.
For numerical simulations using this approach, see Doikou \textit{et al.} \cite{paper3}.
It would be desirable, of course, to be able to carry out the above process
starting from arbitrary initial data $g_0=g_0(x)$ instead.
This is achievable in principle by employing classical `scattering' methods; 
see McKean~\cite[p.~238]{mckean} for a pertinent remark in this direction,
as well as Discussion Section~\ref{sec:discussion}.
\end{remark}
As for the NLS equation case, different consistent choices
for $\p$ generate matrix versions of the reverse space-time nonlocal mKdV equation,
the real or complex versions, given in Ablowitz and Musslimani~\cite[Eq.'s~(10),~(9)]{AM}.
Naturally we actually generate the matrix kernel versions of these equations. 
\begin{corollary}[Reverse space-time matrix nonlocal mKdV equation]
If we choose $\p(x,t)=-P^{\mathrm{T}}(-x,-t)$ and $\n=-1$,
then the kernel function $g=g(y,z;x,t)$ corresponding to $G$,
for every $t\in[0,T]$ satisfies the reverse space-time real matrix nonlocal kernel mKdV equation:
\begin{align*}
\partial_t g(y,z;x,t) + \partial_x^3g(y,z;x,t)=&\;
3g(y,0;x,t)g^{\mathrm{T}}(0,0;-x,-t)\partial_x g(0,z;x,t)\\
&\;+3(\partial_x g(y,0;x,t))g^{\mathrm{T}}(0,0;-x,-t)g(0,z;x,t).
\end{align*} 
Setting $y=z=0$ generates the reverse space-time real matrix nonlocal mKdV equation.
Similarly if we choose $\p(x,t)=-P^\dag(-x,-t)$ and $\n=-1$, then
the integral kernel $g=g(y,z;x,t)$ corresponding to $G$ satisfies the
corresponding reverse space-time complex matrix nonlocal kernel mKdV equation. And
setting $y=z=0$ generates the reverse space-time complex matrix nonlocal mKdV equation.
\end{corollary}
\begin{proof}
Recall with $\n=-1$ the operator $P=P(x,t)$ satisfies the linear
PDE $\pa_tP+\pa_x^3P=0$ while $\p=\p(x,t)$ satisfies $\pa_t\p+\pa_x^3\p=0$.
Note, by a simple transformation of coordinates,
the choice $\p(x,t)=-P^{\mathrm{T}}(-x,-t)$ is consistent with these two 
linear PDEs. Recall $G=PU$ while $\tilde G=\p V$ where $U=(\id-\p P)^{-1}$
and $V=(\id-P\p)^{-1}$. We observe, substituting for $\p(x,t)=-P^{\mathrm{T}}(-x,-t)$, we have
\begin{equation*}
\tilde G(x,t)=-P^{\mathrm{T}}(-x,-t)\bigl(\id-P(x,t)P^{\mathrm{T}}(-x,-t)\bigr)^{-1},
\end{equation*}
while,
\begin{equation*}
G(-x,-t)=P(-x,-t)\bigl(\id-P^{\mathrm{T}}(x,t)P(-x,-t)\bigr)^{-1}.
\end{equation*}
As previously, we observe $\tilde G(x,t)=-G^{\mathrm{T}}(-x,-t)$.
If we substitute this form for $\tilde G$ into the final relation in Step~5 in
the proof of Theorem~\ref{3rd order}, we arrive at the required result.
The complex nonlocal version for the matrix kernel mKdV equation follows
analogously.
\end{proof}
We now examine the relation between the solutions of the matrix kernel mKdV and KdV equations:
the Miura transformation.
Suppose the operator $P$ has a square-matrix symmetric kernel and satisfies $\pa_tP+\pa_x^3P=0$.
From Theorem~\ref{3rd order} we see the solutions
to the square-matrix symmetric kernel mKdV equation, $[G^{\mathrm{mKdV}}]$,
and the square-matrix symmetric kernel primitive KdV equation, $[G^{\mathrm{KdV}}]$, are respectively given by
\begin{equation*}
G^{\mathrm{mKdV}}\coloneqq P(\id-P^2)^{-1}\quad\text{and}\quad G^{\mathrm{KdV}}\coloneqq P(\id-P)^{-1}.
\end{equation*}
The following Corollary asserts the Miura transformation is essentially a consequence
of the operator decomposition
\begin{equation*}
(\id-P^2)=(\id-P)(\id+P).
\end{equation*}
This naturally applies in the non-commutative setting. Note Theorem~\ref{3rd order} outlines the
linearisation procedure for the primitive form of the square-matrix KdV equation,
and so the solution to the square-matrix KdV equation itself is $[G_x^{\mathrm{KdV}}]$. 
\begin{corollary}[Miura transformation]
Assume the operator $P$ has a square-matrix symmetric kernel
and satisfies $\pa_tP+\pa_x^3P=0$. The solutions 
to the square-matrix symmetric kernel mKdV equation, $[G^{\mathrm{mKdV}}]$,
and the square-matrix symmetric kernel KdV equation, $[G_x^{\mathrm{KdV}}]$,
given in Theorem~\ref{3rd order} are related by the Miura transformation:
\begin{equation*}
[G_x^{\mathrm{KdV}}]=[G_x^{\mathrm{mKdV}}]+[G^{\mathrm{mKdV}}]^2.
\end{equation*}
\end{corollary}
\begin{proof}
For convenience we set $U^{\mathrm{mKdV}}\coloneqq (\id-P^2)^{-1}$
and $U^{\mathrm{KdV}}\coloneqq (\id-P)^{-1}$ and recall from
Theorem~\ref{3rd order} and Definition~\ref{appliedlinear}
that $G^{\mathrm{mKdV}}\coloneqq PU^{\mathrm{mKdV}}$. 
Recall further from the proof of Theorem~\ref{3rd order} that
$G_x^{\mathrm{KdV}}=U_x^{\mathrm{KdV}}$. Further note we have
$U_x^{\mathrm{mKdV}}=U^{\mathrm{mKdV}}(P^2)_xU^{\mathrm{mKdV}}$.
Also in this symmetric matrix kernel scenario we have 
$G^{\mathrm{mKdV}}=PU^{\mathrm{mKdV}}=U^{\mathrm{mKdV}}P$.
With these in hand we see, using the decomposition $(\id-P^2)=(\id-P)(\id+P)$, we have
\begin{align*}
&&(\id-P)^{-1}&=(\id+P)(\id-P^2)^{-1}\\
\Leftrightarrow&& U^{\mathrm{KdV}}&=(\id+P)U^{\mathrm{mKdV}}\\
\Rightarrow&& U_x^{\mathrm{KdV}}&=\bigl(PU^{\mathrm{mKdV}}\bigr)_x+U_x^{\mathrm{mKdV}}\\
\Leftrightarrow&& U_x^{\mathrm{KdV}}&=\bigl(PU^{\mathrm{mKdV}}\bigr)_x+U^{\mathrm{mKdV}}(P^2)_xU^{\mathrm{mKdV}}.
\end{align*}
If we now apply the kernel bracket $[\,\cdot\,]$ and use the kernel bracket product rule
on the second term on the right, then using the identities just outlined, we generate
the Miura transformation stated.
\end{proof}

\subsection{Mix and match}\label{sec:mixandmatch}
Can we combine the cases in Sections~\ref{NLS} and \ref{mKdV}?
Indeed we can and we explore this herein. We assume $\tm=-\m\in\CC$ and $\tn=\n\in\CC$,
though both are in general non-zero. Recall the general application linear system
from Definition~\ref{appliedlinear}. The general result is as follows.
\begin{proposition}[Combined degree-three system]
Assume the Hilbert--Schmidt operators $P$, $\p$, $Q$, $\tilde{Q}$, $G$ and $\tilde{G}$
satisfy the application linear system in Definition~\ref{appliedlinear}
and their corresponding kernels satisfy the assumptions of Lemma~\ref{existuniqpde}.
Assume $\tm=-\m\in\CC$ and $\tn=\n\in\CC$.
Then, for some $T>0$, the integral kernel $g=g(y,z;x,t)$ corresponding to $G$,
for every $t\in[0, T]$ satisfies, when $\p=-P^\dag$, the matrix kernel equation:
\begin{align*}
(\pa_t-\mu_1\pa_x^2-\mu_2\pa_x^3)g(y,z;x,t)
=&-2\mu_1g(y,0;x,t)g^\dag(0,0;x,t)g(0,z;x,t)\\
&-3\mu_2g(y,0;x,t)g^\dag(0,0;x,t)\partial_x g(0,z;x,t)\\
&-3\mu_2(\partial_x g(y,0;x,t))g^\dag(0,0;x,t)g(0,z;x,t).
\end{align*} 
In particular, $\langle G\rangle(x,t)\coloneqq g(0,0;x,t)$ satisfies the matrix equation:
\begin{multline*}
(\pa_t-\mu_1\pa_x^2-\mu_2\pa_x^3)\langle G\rangle\\
=-2\mu_1\langle G\rangle)\langle G\rangle^\dag\langle G\rangle
-3\mu_2\langle G\rangle\langle G\rangle^\dag\partial_x\langle G\rangle
-3\mu_2(\partial_x\langle G\rangle)\langle G\rangle^\dag\langle G\rangle.
\end{multline*} 
\end{proposition}
\begin{proof}
We proceed as previously, though relatively quickly we can simply rely on
results already established in the proofs of Theorems~\ref{2nd order} and \ref{3rd order}.
Recall the system of equations in Definition~\ref{appliedlinear}.
With $G=PU$, using the Leibniz rule, that $P_t=\n P_{xxx}$
and the identities for $\pa U$, $U_{xx}$
and $U_{xxx}$ from Lemma~\ref{Uid} in (i), (iv) and (v), we find
\begin{align*}
\pa_t G-\m&\pa_x^2G-\n\pa_x^3 G\\
=&\; P_tU -PUQ_tU-\m\bigl(P_{xx}U+2P_xU_x+PU_{xx}\bigr)\\
&\;-\n\bigl(P_{xxx}U+3P_{xx}U_x+3P_xU_{xx}+PU_{xxx}\bigr)\\
=&\; -PU(Q_t-\m Q_{xx}-\n Q_{xxx})U+2\m\bigl(P_xUQ_xU+PU_xQ_xU\bigr)\\
&\;+3\n\bigl(P_{xx}UQ_xU+2P_xU_xQ_xU+P_xUQ_{xx}U\\
&\;+2PU_xQ_xU_x+PUQ_{xx}U_x+PU_xQ_{xx}U\bigr).
\end{align*}
We now set $\tm=-\m$ and $\tn=\n$ and assume this holds hereafter.
Since $Q\coloneqq\p P$, by direct computation we have
\begin{align*}
Q_t-\m Q_{xx}-\n Q_{xxx}
=&\; \bigl(\tm\p_{xx}+\tn\p_{xxx}\bigr)P + \p\bigl(\m P_{xx}+\n P_{xxx}\bigr) \\
&\;-\m\bigl(\p_{xx}P+2\p_x P_x+\p P_{xx}\bigr)\\
&\;-\n\bigl(\p_{xxx}P+3\p{xx}P_x+3\p_xP_{xx}+\p P_{xxx}\bigr)\\
=&\;(\tm-\m)\p_{xx}P+(\tn-\n)\p_{xxx}P\\
&\;-2\m\p_x P_x-3\n(\p_xP_x)_x\\
=&\;-2\m(\p_x P)_x-3\n(\p_{x}P_x)_x.
\end{align*}
Substituting this result into the the previous result above
generates a relation with `$\pa_t G-\m\pa_x^2G-\n\pa_x^3 G$' 
on the left-hand side while on the right-hand side we
have the sum, of the terms on the right in the final
relation in Step~1 in the proof of Theorem~\ref{2nd order},
and the terms on the right (after multiplying through by $3\n$)
in the final relation in Step~1 in the proof of Theorem~\ref{3rd order}.
Consequently as we did in Step~2 of the proofs of both Theorems,
we can apply the kernel bracket and use the kernel bracket product rule
and so forth. We can proceed exactly as we did in both proofs
and at the appropriate stage make the choice $\p=-P^\dag$, where
$P^\dag$ is the operator adjoint to $P$.
Note as indicated in Remarks~\ref{rmk:choicetildep} and \ref{rmk:complexmKdV}
at the very late stage in both proofs where $\p$ is chosen, we
need to slightly modify the proofs as $\p=-P^\dag$ implies
$\tilde{G}=-G^\dag$. Both results stated in the Proposition thus follow,
where $g^\dag$ denotes the complex-conjugate matrix transpose.
\end{proof}

\section{Discussion}\label{sec:discussion}
In this paper we have presented a unified approach to linearise
and thereby solve many matrix-valued integrable systems
with local and nonlocal nonlinearities. We have also shown
that all the evolutionary nonlinear flows we present are
evolutionary Grassmannian flows.
There are however, many interesting issues still requiring
further resolution. We discuss these here.

First, there are many further systems with both local and nonlocal nonlinearities
in particular mentioned in Ablowitz and Musslimani~\cite{AM} which we
would like to investigate to see if they can be incorporated in the linearisation
procedure we advocate herein, or at least some variant of it.
One aspect of such investigations would be to see if the induction argument
used by P\"oppe~\cite{poppe2} to derive the KdV hierarchy can be extended
to the non-commutative case.

Second, a comparison between the Fredholm Grassmannian flows we consider here and
the Fredholm Grassmannian solutions considered by Segal and Wilson~\cite{SW}
generates a plethora of possible interesting investigative avenues. 
Segal and Wilson~\cite{SW} consider solution curves on the Fredholm Grassmannian
correspnding to solutions of the full Korteweg de Vries and indeed
Kadomtsev--Petviashvili (KP) hierarchies. They do only treat the scalar
commutative case though. Indeed P\"oppe's original method also extends to these
hierarchies; see P\"oppe~\cite{poppe2,poppe3} and P\"oppe and Sattinger~\cite{poppe4}.
P\"oppe~\cite{poppe2} uses an ingenious recursion relation in the KdV hierarchy case.
Segal and Wilson develop solutions in terms of the determinant bundle associated with
the underlying Fredholm Grassmannian. P\"oppe also expresses the solution
in such a form. Indeed the solution can be expressed in terms of derivatives of the log
of a Fredholm determinant of a map whose graph is a given coordinate
patch of the underlying Fredholm Grassmannian. Segal and Wilson~\cite[Prop.~3.3]{SW}
and P\"oppe~\cite{poppe2} come to this result by slightly different means.
That the solution of the KdV equation can be expressed in terms of such a determinant is
originally due to Dyson~\cite{dyson}. Segal and Wilson expicitly
connect it to the tau-function. Particular classes of solution
curves are associated with particular submanifolds of the Fredholm Grassmannian.
For example rational solution curves can be identified as such.
The decomposition $\Hb=\Vb\oplus\Vb^{\perp}$ that underlies the Fredholm Grassmannian
considered by Segal and Wilson separates Fourier basis functions.
The corresponding decompostion we consider is determined by the
operator pairing $(\mathcal P,\mathcal Q)$ which are distinguished
in the underlying linear system. Matching and/or determining the mapping
between the two corresponding Grassmannians is very much of interest.
Extending the approach we advocate to the full KP hierarchy is also
very much of interest. Another way in which the approach we advocate
based on P\"oppe's method is distinguished, apart from the fact
that we preferentially operate at the operator/kernel level,
is that we explicitly write down the linear flow underpinning the integrable systems considered.
And we do so in principle for arbitrary initial data.
Recall from Remark~\ref{rmk:IVPs}, there is a small gap in the
procedure to solving initial value problems via the approach we advocate,
which is surmountable by scattering procedures, see McKean~\cite[p.~238]{mckean}.
An efficient practical procedure to bridge this gap is also very much of
interest. At the moment, for example in the case of the KdV or mKdV equations,
given arbitrary initial data $p_0=p_0(x)$ we can analytically solve the 
linear equation $\pa_tp+\pa_x^3p=0$ in order to evaluate the solution
$p=p(x,t)$ at any time $t>0$. We can then solve the Riccati relation
Fredholm equation to determine $g(y,z;x,t)$ which we can specialise
to $g(0,0;x,t)$ for the corresponding classical KdV or mKdV solution.
See for example Doikou \textit{et al.\/} \cite{paper3} for numerical solutions
generated in this way. However in general, we would like to add the step where
we are given $g(0,0;x,0)$ and then generate $p_0$ from that.
This means solving the Riccati relation for $p_0$ which requires knowledge of $g_0(0,z;x,0)$. 

Third, another interesting aspect of the fact the flows we consider are Grassmannian flows,
is that we have derived them as flows in the canonical coordinate patch distinguished
by the projection $(\id+\mathcal Q,\mathcal P)\to(\id,\mathcal G)$. This is only
possible provided the Fredholm or modified Fredholm determinant of `$\id+\mathcal Q$' 
is non-zero. If this determinant does become zero then this induces countably many
components, and eigenvalues, of $\mathcal G$ to become singular; see for example Beck and Malham~\cite{BM}.
Such singularites are a consequence of a poor representative coordinate patch, and we
should simply project the data from the linearised system $(\id+\mathcal Q,\mathcal P)$
onto a different coordinate patch as indicated in Section~\ref{subsec:Grassmannianflow},
which is always possible by construction. Note the underlying linear system is not singular.
Characterising the time evolutionary route to such singular behaviour and the relation to poles and/or
other singularities in the integrable system solution represents a major challenge; though
Segal and Wilson and also P\"oppe have partially addressed this in terms of rational 
solutions. However, to emphasise, the idea is that as we consider the time evolution of the solution
from any initial data, any singularites are an artifact of a poor representative coordinate patch
of the Fredholm Grassmannian flow, and the flow can be straightforwardly propagated beyond
the singularity by swapping to a different representative coordinate patch. Such
changes of coordinate patch would then be invoked as any further singularities are approached. 

Lastly we remark that the Grassmannian flow approach based on P\"oppe's method
we advocate can be abstracted further. Indeed see Bauhardt and P\"oppe~\cite{BP-ZS}
who show how P\"oppe's method can be extended to difference equation versions
of integrable systems. We could in principle abstract the context herein to
a flow on an operator algebra which we endow with a derivation operator and
a product rule, the latter with the properties of the kernel product rule.
In principle such abstraction might help presciently identify integrable systems.


\section*{Acknowledgements}
We thank the referees for their very useful comments which helped to significantly improve the original manuscript.

\end{document}